\documentclass[12pt]{article}
\usepackage{amsmath}
\usepackage{amsthm}
\usepackage{amssymb}
\usepackage{geometry}
\usepackage{xcolor}
\usepackage{enumitem}
\usepackage{booktabs}
\usepackage{graphicx}
\usepackage{tikz}
\usepackage{pgfplots}
\pgfplotsset{compat=1.16}
\usepackage{hyperref}
\usepackage[numbers]{natbib}

\hypersetup{colorlinks=true,linkcolor=blue!50!black,citecolor=blue!50!black,urlcolor=blue!50!black}

\newtheorem{theorem}{Theorem}
\newtheorem{proposition}[theorem]{Proposition}
\newtheorem{lemma}[theorem]{Lemma}
\newtheorem{corollary}[theorem]{Corollary}
\theoremstyle{definition}
\newtheorem{definition}[theorem]{Definition}
\newtheorem{assumption}[theorem]{Hypothesis}
\theoremstyle{remark}
\newtheorem{remark}[theorem]{Remark}
\newtheorem{example}[theorem]{Example}
\newtheorem{problem}[theorem]{Problem}

\newcommand{\Pp}{\mathbb{P}}
\newcommand{\E}{\mathbb{E}}
\DeclareMathOperator{\Var}{Var}
\newcommand{\Unif}{\mathrm{Unif}}
\newcommand{\iid}{\stackrel{\text{iid}}{\sim}}
\newcommand{\ignore}[1]{}
\DeclareMathOperator{\Inf}{Inf}
\newcommand{\Ent}{\mathrm{Ent}}        

\makeatletter
\ifdefined\pdfextension
  \immediate\pdfextension obj stream attr {/Type /Metadata /Subtype /XML} file {\jobname.xmp}
  \pdfextension catalog {/Metadata \the\numexpr\pdffeedback lastobj\relax\space 0 R}
\else\ifdefined\pdfobj
  \immediate\pdfobj stream attr {/Type /Metadata /Subtype /XML} file {\jobname.xmp}
  \pdfcatalog{/Metadata \the\pdflastobj\space 0 R}
\fi\fi
\makeatother

\begin{document}
\title{Consensus times for monotone aggregation dynamics}
\author{Elchanan Mossel}
\date{\today}
\maketitle

\begin{abstract}
We study an asynchronous consensus dynamics on $N$ agents: at each step
a uniformly chosen agent replaces its state by $f(Y_1,\dots,Y_r)$, where
$f$ is a fixed monotone aggregation rule and $Y_1,\dots,Y_r$ are the
states of $r$ agents sampled uniformly with replacement.  Let $T$ be the
first time at which all agents agree.

Let $f:\{0,1\}^r\to\{0,1\}$ be monotone and non-constant.  The expected consensus time is governed by the two
\emph{endpoint degrees} $D_0(f)=\#\{i:f(e_i)=1\}$ and
$D_1(f)=\#\{i:f(\mathbf{1}-e_i)=0\}$, where $e_i$ is the $i$th standard
basis vector and $\mathbf{1}$ the all-ones vector.  
If $D_0(f)\ne 1$ and
$D_1(f)\ne 1$, then $\E[T]=\Theta\bigl(N(1+\log d_f)\bigr)$, where
$d_f$ counts the agents that must change state before the nearest
attracting consensus is reached, so that $\E[T]=O(N\log N)$
uniformly over initial states.  If $f$ is a dictator, $\E[T]$ is given
by a voter-model formula and equals
$\Theta\bigl(N^2\Ent(p_0)\bigr)$ with absolute constants,
uniformly over the initial state, where $\Ent$ is the binary entropy
and $p_0$ the initial fraction of agents in state $1$.  Otherwise
exactly one of $D_0(f),D_1(f)$ equals $1$ and the other equals $0$.
The worst-case expected consensus time is then $\Theta(N^{2-1/m})$, where
$m\ge 2$ is the least number of coordinates that force the value $1$ in
the residual rule of $f$, defined in Section~\ref{sec:intro} (of its
dual, when $D_1(f)=1$).  
\footnote{\textcopyright{} 2026 Elchanan Mossel. The author expressly
reserves the use of this work for text and data mining under
Article~4(3) of Directive (EU) 2019/790, including for the training or
fine-tuning of artificial intelligence models and the generation of
embeddings or synthetic derivatives, and this reservation is directed
to providers of general-purpose AI models under Article~53(1)(c) of
Regulation (EU) 2024/1689. Rights under Article~3 of Directive (EU)
2019/790 are unaffected.}


\end{abstract}

\noindent\textbf{MSC 2020.} 60J10, 60K35, 91D30, 68W15.

\smallskip
\noindent\textbf{Keywords.} consensus time, monotone Boolean function,
voter model, birth-death chain, scale function, Poincar\'e inequality,
opinion dynamics.

\section{Introduction}
\label{sec:intro}

\subsection{Model and Main Results}
We begin with a formal definition of the model and a statement of the main results. 
\paragraph{The dynamics.}
Let $r\ge 1$ and let $f:\{0,1\}^r\to\{0,1\}$ be coordinate-wise monotone
and \emph{idempotent}, meaning $f(0,\dots,0)=0$ and $f(1,\dots,1)=1$.
We study the discrete-time Markov chain
$X^{(t)}=(X_1^{(t)},\dots,X_N^{(t)})\in\{0,1\}^N$ defined by the following
step.
\begin{enumerate}[label=(\arabic*),itemsep=0pt]
\item Draw $V\sim\Unif([N])$, the agent that updates.
\item Draw $J_1,\dots,J_r\iid\Unif([N])$, independent of $V$.
\item Set $X_V^{(t+1)}=f(X_{J_1}^{(t)},\dots,X_{J_r}^{(t)})$ and leave
every other coordinate unchanged.
\end{enumerate}
Sampling is with replacement, so the $J_u$ may coincide with each
other and with $V$.  Idempotence makes the two constant configurations
absorbing, and we call
$T=\inf\{t:X^{(t)}\text{ is constant}\}$ the \emph{consensus time}.
Write $\mathcal{F}_t=\sigma(X^{(0)},\dots,X^{(t)})$,
$S_t=\#\{i:X_i^{(t)}=1\}$, $S_0=s$ and $p_t=S_t/N$, and set
$g(p)=\E_{\mathrm{Bern}(p)^{\otimes r}}[f]$, the \emph{mean-field map},
and $h=g-\mathrm{id}$.

We fix two more pieces of notation.  For
$i\in[r]$, $e_i$ is the
$i$th standard basis vector of $\{0,1\}^r$, $\mathbf{1}$ is the all-ones
vector and $\mathbf{0}$ the all-zeros vector.  A monotone
$f:\{0,1\}^r\to\{0,1\}$ is a \emph{dictator} if $f(x)=x_j$ for some
fixed $j\in[r]$.

\paragraph{Notation.}
Let $f:\{0,1\}^r\to\{0,1\}$ be monotone and non-constant, and let
\[
D_0(f)=\#\{i:f(e_i)=1\}=g'(0),\qquad
D_1(f)=\#\{i:f(\mathbf{1}-e_i)=0\}=g'(1)
\]
be its \emph{endpoint degrees} (the identifications with $g'(0)$ and $g'(1)$
are proved in Lemma~\ref{lem:g-binary}).  An absorbing endpoint is
\emph{attracting} when the drift $h=g-\mathrm{id}$ points towards it
in a neighbourhood; this holds for $0$ exactly when $D_0(f)=0$ and for $1$
exactly when $D_1(f)=0$.  Write $d_f(s)$ for the number of agents that must
change state before the nearest attracting consensus is reached
(Definition~\ref{def:attracting}), and call $f$ \emph{non-degenerate} when
$D_0(f)\ne 1$ and $D_1(f)\ne 1$.

For a degenerate non-dictator with $D_0(f)=1$, say $f(e_j)=1$, one has
$f(x)=x_j\vee\tilde f(x_{-j})$ where $\tilde f(x_{-j}):=f(x^{j\to0})$ is
the \emph{residual rule} of $f$ (Lemma~\ref{lem:degstruct}).  A
\emph{minterm} of a monotone Boolean function is a minimal set $B$ of
coordinates such that setting the coordinates in $B$ to $1$ and the rest to
$0$ gives the value $1$; every minterm of $\tilde f$ has size at least $2$.
Define
\begin{equation}\label{eq:mdef}
m = m(f) \;:=\; \min\bigl\{|B| : B \text{ is a minterm of } \tilde f\bigr\} \;\ge\; 2,
\end{equation}
and symmetrically (replacing $f$ by its dual $f^*$) when $D_1(f)=1$.

\begin{theorem}[Main result]\label{thm:main}
Let $f:\{0,1\}^r\to\{0,1\}$ be monotone and non-constant.  Every such $f$
falls into exactly one of three regimes.

\emph{(i) Non-degenerate: $D_0(f)\ne1$ and $D_1(f)\ne1$.}
There are constants $0<c<C<\infty$ depending only on $f$ such that
\[
cN\bigl(1+\log d_f(s)\bigr)\;\le\;\E[T\mid S_0=s]\;\le\;
CN\bigl(1+\log d_f(s)\bigr)
\qquad\text{for all }1\le s\le N-1.
\]
In particular $\E[T]=O(N\log N)$ uniformly in the initial state, and
$\E[T]=\Theta(N\log N)$ whenever $d_f(s)\ge\varepsilon N$.

\emph{(ii) Dictator: $D_0(f)=D_1(f)=1$.}
If $f(x)=x_j$ the chain is the asynchronous voter model on the complete
graph.  With $p_0=s/N$ and $\Ent(p)=-p\log p-(1-p)\log(1-p)$,
\[
\E[T\mid S_0=s]\;=\;\Theta\bigl(N^2\Ent(p_0)\bigr)
\]
with absolute constants, uniformly over $N\ge2$ and $1\le s\le N-1$
(Theorems~\ref{thm:bin-dict} and~\ref{cor:dict-asymp}).

\emph{(iii) Degenerate non-dictator: $\{D_0(f),D_1(f)\}=\{0,1\}$.}
With $m=m(f)\ge 2$ as in~\eqref{eq:mdef},
\[
\max_{1\le s\le N-1}\E[T\mid S_0=s]\;=\;\Theta\bigl(N^{2-1/m}\bigr),
\]
the lower bound being attained at $s\asymp N^{(m-1)/m}$.
\end{theorem}

In some sense these results say that the speed of convergence is determined by how dictatorial the function is: slowest for dictators, less slow for non-dictatorial functions that are dictator-like near $0$ or $1$, and fastest for all other functions.


\subsection{AI Use}
This paper is based on extensive collaboration with Claude. 
It is first motivated by the result of \cite{HA} that led the autor to consider the dynamics of this paper and conjecture that 
all monotone non-dictatorial functions have $O(n \log n)$ absorption time. The author then provided to Claude a sketch of the proof based on the analysis of the mean-field equation, Poincare inequality and martingale techniques. 
At first, Claude provided a proof along these lines as both the author and Claude did not consider the subtleties of the singularities of the mean field equation at $0$ and $1$. 
Later when the author requested extensive adversarial checks, Claude found an example with converge time of order $n^{3/2}$ which prompted Claude to suggest replacing some martingale arguments by more explicit Green function estimates.
The author together with Claude organized and cleaned this proof and also provided a high-level sketch of the continuous picture that underlies the results. 
The author read all the proofs in the paper and if there are errors, the responsibility lies with the author.

\subsection{Proof ideas}
The proof reduces the dynamics to a birth-death chain:
$S_t$ is itself a Markov chain on $\{0,\dots,N\}$, with $0$ and $N$
absorbing and with
\[
\Pp(i\to i+1)=(1-\pi)g(\pi),\qquad \Pp(i\to i-1)=\pi(1-g(\pi)),
\qquad \pi=i/N .
\]
Write $\beta_i$ and $\delta_i$ for these birth and death
probabilities.  The scale function
$\gamma_i=\prod_{l=1}^{i}\delta_l/\beta_l$ and its partial sums
$\Sigma_i=\sum_{l<i}\gamma_l$ give the expected absorption time of a
birth-death chain in closed form (Proposition~\ref{prop:green}).

Two structural inputs drive the estimates.  First, the Bernoulli
Poincar\'e inequality $\Var_p(f)\le p(1-p)g'(p)$, applied at an interior
fixed point $p_*$ where $\Var_{p_*}(f)=p_*(1-p_*)$, gives
$g'(p_*)\ge 1$, with equality only for constants and dictators.  Hence
every interior fixed point is repelling, there is at most one of them,
and $h$ has exactly one of three sign patterns on $(0,1)$
(Corollary~\ref{cor:sign}).  Second, a combinatorial lemma identifies
$g'(0)$ and $g'(1)$ with the endpoint degrees $D_0(f),D_1(f)$ and
matches each sign pattern to its attracting endpoints
(Lemma~\ref{lem:endpoint}).

By Proposition~\ref{prop:green}, $\varphi(s)=\sum_j G(s,j)$, where
$G(s,j)$ is the expected number of steps the chain spends at $j$ before
absorption, starting from $s$; since lingering at $j$ is easiest when
starting there, $G(s,j)\le G(j,j)=:\widehat G(j)$ for every $s$, an
$s$-independent envelope.

Away from an endpoint of degree exactly $1$, $R=\delta/\beta$ has a
limit there that is never $1$: a finite value $1/D_0$ or $D_1$ when the
endpoint is repelling (degree $\ge2$), or $0$ or $+\infty$ when it is
attracting (degree $0$) (Lemma~\ref{lem:rates}).  Continuity then bounds
$R$ away from $1$ on the whole closed interval whenever neither endpoint
has degree $1$, so $\gamma$ is a globally geometric sequence -- uniformly
decaying if $D_1=0$, uniformly growing if $D_0=0$
(Lemma~\ref{lem:scale}(a),(b)) -- and $\widehat G(j)\asymp N^2/(j(N-j))$
at every state $j$.  What makes this the actual size of $\varphi(s)$,
not merely a bound on it, is that $G(s,j)$ itself attains this order,
not just its envelope: for every $j$ between $s$ and its nearer
attracting endpoint the damping factor
$\rho_s(j)=G(s,j)/\widehat G(j)$ of~\eqref{eq:decomp} stays $\asymp1$
over that whole stretch, of length $d_f(s)$, and summing
$N^2/(j(N-j))\asymp N/\min(j,N-j)$ there gives $N\log d_f(s)$.  On the
other side of $s$, $\rho_s(j)$ decays geometrically in $|j-s|$ -- the
same boundedness of $R$ away from $1$, now read as a gambler's-ruin
ratio rather than a local rate -- so that side contributes only $O(N)$
however far it runs.  Altogether $\varphi(s)\asymp N(1+\log d_f(s))$,
and in particular $\E[T]=O(N\log N)$ uniformly.

The interior fixed point $p_*$, when both endpoints are attracting, is
governed by the same dichotomy at one remove.  The strict Bernoulli
Poincar\'e inequality forces $\log R$ to have a simple, first-order zero
at $p_*$, so $\log\gamma_j=\sum_{l\le j}\log R(l/N)$ has a genuine
quadratic peak at $j=p_*N$ (Figure~\ref{fig:gamma}): $\gamma$ stays
within a constant factor of that peak only over a window of width
$O(\sqrt N)$, with $R$ again bounded away from $1$ outside it
(Lemma~\ref{lem:scale}(c)).  Summed by the same dyadic argument, now in
distance from the peak rather than from an endpoint, the window
contributes the same order $N\log N$ as the rest of the range and so
never enlarges the exponent; its role is instead to decouple the two
attracting basins (Remark~\ref{rem:poincare-role}).

The one way this boundedness can fail is an endpoint degree of exactly
$1$: then $R\to1$ \emph{at} that endpoint, at the polynomial rate
$R=1-\Theta(\pi^{m-1})$ fixed by the least minterm size $m\ge2$ of the
residual rule (Lemma~\ref{lem:degstruct}), rather than approaching a
limit $\ne1$.  Now $\gamma_j\asymp1$, with no decay at all
(Lemma~\ref{lem:degprofile}(i),(ii)), until $j$ is large enough that the
accumulated exponent $N(j/N)^m$ reaches order $1$, i.e.\ up to the
crossover $K=N^{(m-1)/m}$; past $K$, $\gamma_j$ falls off as
$\exp(-\Theta(j^m/N^{m-1}))$, with $M_j\le C(N/j)^{m-1}$
(Lemma~\ref{lem:degprofile}(iv)).  Both the flat plateau below $K$ and
the falling tail just above it contribute the same order $N^{2-1/m}$,
concentrated within a single range around $j\asymp K$ rather than
spread over the $O(\log N)$ dyadic scales that produced the logarithm
above: this is regime~(iii), with the worst initial state at
$s\asymp K$ (Theorem~\ref{thm:bin-deg}).

\subsection{Related work}
The mean-field equation $p'=g(p)-p$ for specific families of $f$
has been analysed in the statistical-physics literature: the nonlinear
$q$-voter model of Castellano, Mu\~noz and Pastor-Satorras
\cite{CMP09}, the Galam majority-rule models \cite{Galam}, and the
textbook treatment of Krapivsky, Redner and Ben-Naim \cite{KRBN10}.
Sood and Redner \cite{SR05} treat heterogeneous-graph extensions.

The birth-death reduction and its Green's-function machinery are
themselves classical: they go back to Karlin and McGregor's
construction of the scale function and speed measure of a
birth-death chain \cite{KM57}.  The same visit-counting idea recurs
in two more recent papers whose setting is close to ours.  Hathcock
and Strogatz \cite{HS22} decompose a birth-death generator into a
biased-random-walk part and a diagonal waiting-time part and read
off a matrix of expected visit numbers -- our $G(s,j)$, in different
notation -- to obtain absorption-time distributions for a broad
class of extinction-prone chains, including a one-parameter family
whose rates vanish at a boundary at a tunable polynomial rate; the
resulting universality classes play the same role there that the
minterm size $m$ plays in our Lemma~\ref{lem:degprofile}.  Altrock
and Traulsen \cite{AT09} study weak-selection fixation times for the
Fermi process using the identical ratio $\gamma_i=T_i^-/T_i^+$ that
we call $\gamma$, though their expansion is regular rather than
singular.  Closer still to our degenerate regime~(iii): Doering,
Sargsyan and Sander \cite{DSS04} show, for a related class of
birth-death processes, that the naive Fokker--Planck (mean-field
diffusion) limit of the master equation reproduces the exact
absorption time only when the deterministic drift is small
everywhere, and that at a boundary where the drift's derivative also
vanishes -- the analogue of our $\Psi'(0)=0$ -- the exact chain
acquires an anomalous power-law correction that the naive continuum
equation misses; we return to this parallel in
Remark~\ref{rem:mfe-singularity}.

In distributed computing, Doerr, Goldberg, Minder, Sauerwald and
Scheideler \cite{DGMSS11}, Cooper, Els\"asser and Radzik \cite{CER14},
Becchetti, Clementi, Natale, Pasquale and Silvestri \cite{BCNPS15},
Berenbrink, Clementi, Els\"asser, Kling, Mallmann-Trenn and Natale
\cite{BCEKMN17}, and Ghaffari and Lengler \cite{GL18} prove
$O(\log N)$ consensus bounds for $2$-choice, $3$-majority and plurality
dynamics on $K_N$ and on random graphs.  These are results for the
\emph{synchronous} model, in which every agent updates in every round.
A synchronous bound of $O(\log N)$ rounds is often quoted as an
asynchronous bound of $O(N\log N)$ steps, on the grounds that $N$
asynchronous steps make up one round on average.  
We note this 
correspondence at the heuristic level as 
the two models are not equivalent, i.e., synchronous majority dynamics
admits cycles of period two, which never reach consensus, whereas the
asynchronous chain absorbs almost surely from every initial state.
Mohan and Pra\l at \cite{MP24}
analyse \emph{asynchronous} majority dynamics on binomial random graphs,
where one agent updates at a time.  Among the works listed here, that is
the closest in timing to ours, though the update rule there is majority
over a fixed neighbourhood rather than over sampled agents.

Mossel, Neeman and Tamuz \cite{MNT14} and the Mossel--Tamuz survey
\cite{MT17} treat iterative majority on social networks from an
information-aggregation viewpoint.  Tran and Vu \cite{TV20} study synchronous
majority dynamics on $G(n,p)$ and prove a ``power of few'' phenomenon:
if one of the two opinions holds an initial lead of at least a constant
number of agents, it wins with probability at least $1-\varepsilon$.
Unanimity is then reached after a constant number of rounds, four in
their sharpest statement.  Their technique is a shrinking argument,
which bounds the size of the minority set after each round through a
universal-reduction property of the graph.  Tran and Vu \cite{TV25}
extend the phenomenon to the sparse regime.  Their setting differs from
ours in three ways: the update rule is majority over $\Theta(n)$
neighbours rather than over $r$ sampled agents, the timing is
synchronous rather than asynchronous, and the quantity studied is which
opinion wins from a near-balanced start rather than the expected time to
consensus.  Their results therefore do not specialise to any case
treated here.  Mossel and Schoenebeck \cite{MS10} formulate the problem
with memory and communication constraints.

The Bernoulli Poincar\'e step in our proof is the same ingredient
that drives the sharp-threshold theorems of Russo \cite{Russo82},
Friedgut and Kalai \cite{FK96}, and
Bourgain--Kahn--Kalai--Katznelson--Linial \cite{BKKKL92}; see also
Rossignol \cite{Rossignol06} and Kalai's survey \cite{Kalai18}.
The mean-field map $g$ is monotone, so the flow $p'=g(p)-p$ is a
monotone dynamical system in the Hirsch--Smith sense
\cite{HirschSmith05}; monotone couplings for attractive interacting
particle systems go back to the Liggett--Holley device
\cite{Liggett85}.  Related mean-field
analyses of monotone probabilistic cellular automata appear in Balister,
Bollob\'as and Kozma \cite{BBK06}, and bootstrap percolation
\cite{Holroyd03,BBDM12,JLTV12} is the one-sided analogue of the
dynamics studied here.

Our paper is directly motivated by the recent preprint 
\emph{Global Stability of Coordination under Monotone Sampling} by Heller and Arigapudi \cite{HA} where similar dynamics are studied for an infinite population model and where among monotone functions it is shown that dictator is the only one that is not {\em stable}.


\subsection{Outline}
Section~\ref{sec:binary} proves the trichotomy and treats the
non-degenerate and the degenerate regimes.  Section~\ref{sec:dict}
treats the dictator case as the flat special case of the birth-death
formula.

\subsection*{Acknowledgments}
The author was partially supported by  Bush Faculty Fellowship ONR-N00014-20-1-2826, Simons Investigator award (622132), and MURI  grant N000142412742. 

\section{The trichotomy and the non-dictator regimes}
\label{sec:binary}

Throughout, $f:\{0,1\}^r\to\{0,1\}$ is monotone and non-constant.  Define the two \emph{endpoint degrees}
\begin{equation}\label{eq:endpointdeg}
D_0(f)\;=\;\#\{i:f(e_i)=1\}\;=\;g'(0),\qquad
D_1(f)\;=\;\#\{i:f(\mathbf 1-e_i)=0\}\;=\;g'(1),
\end{equation}
the identifications with the endpoint derivatives of $g$ being part of
Lemma~\ref{lem:g-binary}, and call $f$ \emph{non-degenerate} if
$D_0(f)\ne1$ and $D_1(f)\ne1$.  These two counts partition the monotone
non-constant rules into three branches
(Theorem~\ref{thm:trichotomy} below).  The consensus time is determined
for the dictator branch in Section~\ref{sec:dict} and for the other two
here; the results are collected in Table~\ref{tab:trichotomy}.

\subsection{\texorpdfstring{Preliminaries on $g$ and the Bernoulli Poincar\'e inequality}{Preliminaries on g and the Bernoulli Poincare inequality}}

We begin with some well known preliminaries. 
Being monotone and non-constant is equivalent to being monotone and
idempotent.  We keep the notation of Section~\ref{sec:intro}
for $e_i$, $\mathbf 0$, $\mathbf 1$ and dictators, and write $x_{-i}$ for
$x$ with the $i$th coordinate deleted and
$x^{i\to b}$ for $x$ with the $i$th coordinate set to $b$.  The discrete
derivative and the $p$-biased influence are
\[
\partial_i\psi(x)\;=\;\psi(x^{i\to 1})-\psi(x^{i\to 0}),\qquad
\Inf_i^p(f)\;=\;\Pp_p\bigl[f(x^{i\to1})\neq f(x^{i\to0})\bigr],
\]
where $\Pp_p$ is $\mathrm{Bern}(p)^{\otimes r}$; here $\partial_i\psi$ does
not depend on $x_i$, and $\Inf_i^p(f)=\E_p[\partial_i f]$ for monotone
Boolean $f$.

\begin{lemma}[Properties of $g$]\label{lem:g-binary}
Let $f:\{0,1\}^r\to\{0,1\}$ be monotone and non-constant.  Then
$f(\mathbf 0)=0$ and $f(\mathbf 1)=1$; $g$ is a polynomial of degree at most
$r$ with $g(0)=0$ and $g(1)=1$; $g$ is strictly increasing on $[0,1]$ with
$g'(p)>0$ for every $p\in(0,1)$; and
\[
g'(p)\;=\;\sum_{i=1}^r\Inf_i^p(f),\qquad
g'(0)\;=\;\#\{i:f(e_i)=1\},\qquad
g'(1)\;=\;\#\{i:f(\mathbf 1-e_i)=0\}.
\]
The two endpoint derivatives may vanish, so $g'$ is not bounded below by a
positive constant on the closed interval.
\end{lemma}

\begin{proof}
Monotonicity gives $f(\mathbf 0)\le f(x)\le f(\mathbf 1)$ for all $x$.  If
$f(\mathbf 0)=1$ then $f\equiv 1$ and if $f(\mathbf 1)=0$ then $f\equiv 0$,
both contradicting non-constancy; so $f(\mathbf 0)=0$ and $f(\mathbf 1)=1$.
Expanding, $g(p)=\sum_{x}f(x)p^{|x|}(1-p)^{r-|x|}$ with $|x|=\sum_i x_i$ is
a polynomial of degree at most $r$, and evaluating at $p=0,1$ gives
$g(0)=f(\mathbf 0)=0$, $g(1)=f(\mathbf 1)=1$.

The identity $g'(p)=\sum_i\Inf_i^p(f)$ is the Margulis--Russo formula
\cite{Russo82}; for monotone $f$ it also follows from
$\frac{\partial}{\partial p}\E_p[f]=\sum_i\E_p[\partial_i f]$ for a product
measure together with $\E_p[\partial_i f]=\Inf_i^p(f)$.  Since $\Pp_0$ is
the point mass at $\mathbf 0$ and $\Pp_1$ the point mass at $\mathbf 1$,
$\Inf_i^0(f)=\mathbf 1\{f(e_i)=1\}$ and
$\Inf_i^1(f)=\mathbf 1\{f(\mathbf 1-e_i)=0\}$, which gives the two endpoint
formulas.

Since $f$ is non-constant there are $x\le y$ differing in one coordinate $i$
with $f(x)=0<1=f(y)$; for $p\in(0,1)$ every point of $\{0,1\}^r$ has
positive $\Pp_p$-mass, so $\Inf_i^p(f)>0$ and hence $g'(p)>0$, and $g$ is
strictly increasing on $[0,1]$.  Both endpoint derivatives vanish for
$f=\mathrm{MAJ}_3$, so no positive lower bound on $g'$ holds on $[0,1]$.
\end{proof}

\begin{lemma}[Bernoulli Poincar\'e inequality]\label{lem:poincare}
Let $p\in(0,1)$ and $\psi:\{0,1\}^r\to\mathbb R$.  Then
\[
\Var_p(\psi)\;\le\;p(1-p)\sum_{i=1}^r\E_p\bigl[(\partial_i\psi)^2\bigr],
\]
with equality if and only if $\psi$ has $p$-biased Fourier degree at most
$1$, that is, $\psi(x)=c_0+\sum_{i=1}^r\psi_i(x_i)$ for some functions
$\psi_i$ of a single coordinate.  Among monotone Boolean $\psi$ the
functions of degree at most $1$ are exactly the constants and the
dictators.  In particular if $f$ is monotone Boolean and not a constant or
a dictator, then
$g(p)(1-g(p))<p(1-p)g'(p)$ for every $p\in(0,1)$.
\end{lemma}

\begin{proof}
Put $\sigma=\sqrt{p(1-p)}$ and $\chi_i(x)=(x_i-p)/\sigma$, and for
$B\subseteq[r]$ let $\chi_B=\prod_{i\in B}\chi_i$.  The $\chi_B$ form an orthonormal basis of
$L^2(\Pp_p)$, so $\psi=\sum_B\hat\psi(B)\chi_B$ with
$\hat\psi(B)=\E_p[\psi\chi_B]$, and
$\Var_p(\psi)=\sum_{B\neq\emptyset}\hat\psi(B)^2$.  Since
$\partial_i\chi_i=1/\sigma$ and $\partial_i\chi_B=0$ for $i\notin B$, we get
$\partial_i\chi_B=\chi_{B\setminus\{i\}}/\sigma$ for $i\in B$ and hence
$\partial_i\psi=\sigma^{-1}\sum_{B\ni i}\hat\psi(B)\chi_{B\setminus\{i\}}$.
The functions $\chi_{B\setminus\{i\}}$, $B\ni i$, are orthonormal, so
\[
p(1-p)\sum_{i=1}^r\E_p\bigl[(\partial_i\psi)^2\bigr]
=\sum_{i=1}^r\sum_{B\ni i}\hat\psi(B)^2
=\sum_{B}|B|\,\hat\psi(B)^2 .
\]
Comparing with $\Var_p(\psi)=\sum_{B\neq\emptyset}\hat\psi(B)^2$ gives the
inequality, with equality if and only if $\hat\psi(B)=0$ whenever
$|B|\ge 2$, that is, if and only if $\psi$ has $p$-biased degree at most $1$.
The change of basis between $\{\chi_B\}$ and the monomials
$\{\prod_{i\in B}x_i\}$ is triangular with respect to inclusion, so degree
at most $1$ in the $p$-biased basis is the same as degree at most $1$ in the
multilinear representation, and does not depend on $p$.

Let $\psi$ be monotone Boolean of degree at most $1$, say
$\psi(x)=c_0+\sum_i c_ix_i$.  Then $\partial_i\psi=c_i$, which lies in
$\{0,1\}$ because $\psi$ is Boolean and monotone.  If $c_i=c_j=1$ for
$i\neq j$, pick $x$ with $x_i=x_j=0$; then
$\psi(x+e_i+e_j)=\psi(x)+2\notin\{0,1\}$, a contradiction.  So at most
one $c_i$ is non-zero, and $\psi$ is a constant or a dictator.

For the last statement take $\psi=f$.  Then $\Var_p(f)=g(p)(1-g(p))$ and
$\E_p[(\partial_if)^2]=\E_p[\partial_i f]=\Inf_i^p(f)$ because
$\partial_if\in\{0,1\}$ for monotone Boolean $f$; now apply
Lemma~\ref{lem:g-binary}.
\end{proof}

\subsection{Fixed-point geometry}

Set $h=g-\mathrm{id}$, so $h(0)=h(1)=0$.

\begin{theorem}[Fixed-point repulsion]\label{thm:fp}
Let $f:\{0,1\}^r\to\{0,1\}$ be monotone, non-constant and not a dictator.
Then $h$ has at most one zero $p_*$ in $(0,1)$, and any such zero satisfies
$g'(p_*)>1$.
\end{theorem}

\begin{proof}
If $p_*\in(0,1)$ and $g(p_*)=p_*$, then $\Var_{p_*}(f)=p_*(1-p_*)$, and
Lemma~\ref{lem:poincare} gives
$p_*(1-p_*)<p_*(1-p_*)g'(p_*)$, that is, $g'(p_*)>1$.

Suppose $h$ had two zeros $p_1<p_2$ in $(0,1)$.  By the previous paragraph
$h'(p_i)=g'(p_i)-1>0$, so $h>0$ on some interval $(p_1,p_1+\varepsilon)$ and
$h<0$ on some interval $(p_2-\varepsilon,p_2)$.  Set
\[
p_3\;:=\;\inf\{p\in(p_1,p_2)\,:\,h(p)\le 0\},
\]
a well-defined element of $(p_1,p_2)$: the set is non-empty because $h<0$
just below $p_2$, and $p_3\ge p_1+\varepsilon>p_1$ because $h>0$ on
$(p_1,p_1+\varepsilon)$.  By definition of the infimum $h>0$ on $(p_1,p_3)$, so
$h(p_3)\ge 0$ by continuity; and there are $t_n\downarrow p_3$ with
$h(t_n)\le0$, so $h(p_3)\le 0$.  Hence $h(p_3)=0$, and $p_3$ is an interior
fixed point.  Since $h>0$ on $(p_1,p_3)$ and $h(p_3)=0$, every difference
quotient $(h(p_3)-h(p))/(p_3-p)$ with $p\in(p_1,p_3)$ is non-positive, so
$h'(p_3)\le 0$, that is, $g'(p_3)\le 1$.  This contradicts the first
paragraph applied at $p_3$.  Taking the infimum in this way covers zeros of
even order as well.
\end{proof}

\begin{corollary}[Sign pattern of the drift]\label{cor:sign}
Under the hypotheses of Theorem~\ref{thm:fp}, exactly one of the following
holds:
\begin{enumerate}[label=(\alph*),itemsep=0pt]
\item $h>0$ on $(0,1)$;
\item $h<0$ on $(0,1)$;
\item there is a unique $p_*\in(0,1)$ with $h(p_*)=0$, and $h<0$ on
$(0,p_*)$, $h>0$ on $(p_*,1)$.
\end{enumerate}
\end{corollary}

\begin{proof}
If $h$ has no zero in $(0,1)$ then, $h$ being continuous, the intermediate
value theorem forces $h$ to have constant sign there, which is case~(a) or
case~(b).  Otherwise, by Theorem~\ref{thm:fp} the zero is unique, call it
$p_*$, and $h$ has constant sign on each of $(0,p_*)$ and $(p_*,1)$, again
by the intermediate value theorem.  Theorem~\ref{thm:fp} gives
$h'(p_*)=g'(p_*)-1>0$, so $h<0$ immediately to the left of $p_*$ and $h>0$
immediately to the right; combined with constancy of sign on each side this
is case~(c).  The strict inequality $h'(p_*)>0$ is exactly what excludes the
attracting configuration $h>0$ on $(0,p_*)$, $h<0$ on $(p_*,1)$.
\end{proof}

\subsection{Endpoint structure}

\begin{lemma}[Endpoint structure]\label{lem:endpoint}
Let $f$ be monotone, non-constant and not a dictator.  Then
\begin{enumerate}[label=(\roman*),itemsep=0pt]
\item $D_0\ge 1$ $\iff$ $f\ge x_j$ pointwise for some $j$ $\iff$ $h>0$ on
$(0,1)$ (case~(a) of Corollary~\ref{cor:sign});
\item $D_1\ge 1$ $\iff$ $f\le x_j$ pointwise for some $j$ $\iff$ $h<0$ on
$(0,1)$ (case~(b));
\item at most one of $D_0,D_1$ is non-zero, and $D_0=D_1=0$ is exactly
case~(c).
\end{enumerate}
A dictator has $D_0=D_1=1$.
\end{lemma}

\begin{proof}
(i) If $f(e_j)=1$ and $x_j=1$ then $x\ge e_j$, so $f(x)\ge f(e_j)=1$; thus
$f\ge x_j$.  Conversely $f\ge x_j$ gives $f(e_j)\ge 1$.  Given $f\ge x_j$,
\[
h(p)\;=\;\E_p[f]-\E_p[x_j]\;=\;\E_p[f-x_j]\;=\;\Pp_p[X_j=0,\ f(X)=1].
\]
Since $f$ is not the dictator $x_j$ and $f\ge x_j$, there exists $x$ with
$x_j=0$ and $f(x)=1$, and this point has positive $\Pp_p$-mass for
$p\in(0,1)$; hence $h>0$ on $(0,1)$.  Conversely, if $D_0=0$ then
$g'(0)=0$ and $g(0)=0$, so $g(p)=O(p^2)$ and $h(p)=g(p)-p<0$ for all small
$p>0$, ruling out case~(a).

(ii) Apply (i) to the dual $f^*(x):=1-f(\mathbf 1-x)$, which is monotone,
non-constant and not a dictator, and satisfies $D_0(f^*)=D_1(f)$ and
$g_{f^*}(p)=1-g(1-p)$, hence $h_{f^*}(p)=-h(1-p)$.

(iii) By (i) and (ii), $D_0\ge1$ forces $h>0$ on $(0,1)$ and $D_1\ge1$
forces $h<0$ there, so they cannot both hold.  If $D_0=D_1=0$ then neither
case~(a) nor case~(b) occurs, so case~(c) does; and conversely case~(c)
excludes (i) and (ii).  For $f=x_j$ one has $f(e_j)=1$, $f(e_i)=0$ for
$i\ne j$, and $f(\mathbf 1-e_j)=0$, $f(\mathbf 1-e_i)=1$ for $i\ne j$.
\end{proof}

\begin{remark}[The endpoint degrees detect a dictator]\label{rem:d0d1}
The last sentence of Lemma~\ref{lem:endpoint} has a converse, in the
strong form: if $D_0(f)\ge1$ and $D_1(f)\ge1$ then $f$ is a dictator.
The argument uses monotonicity alone, so it is independent of
parts~(i)--(iii) and of the machinery behind them.  Suppose $f(e_i)=1$
and $f(\mathbf 1-e_j)=0$.  If $x_i=1$ then $x\ge e_i$, so $f(x)=1$; and
if $x_j=0$ then $x\le\mathbf 1-e_j$, so $f(x)=0$.  Hence
$x_i\le f(x)\le x_j$ for every $x$.  Taking $x=e_i$ gives
$1\le(e_i)_j$, so $j=i$, and then $x_i\le f(x)\le x_i$, that is,
$f=x_i$.  In particular $f$ is a dictator if and only if
$D_0(f)=D_1(f)=1$.
\end{remark}

The trichotomy follows.  The dictator branch is proved in
Section~\ref{sec:dict} and the other two in this section; the times
themselves are collected in Table~\ref{tab:trichotomy}.

\begin{theorem}[Trichotomy]\label{thm:trichotomy}
Every monotone non-constant $f:\{0,1\}^r\to\{0,1\}$ satisfies exactly
one of
\begin{enumerate}[label=(\roman*),itemsep=0pt]
\item $f$ is non-degenerate, that is, $D_0(f)\ne1$ and $D_1(f)\ne1$;
\item $f$ is a dictator, and then $D_0(f)=D_1(f)=1$;
\item $f$ is degenerate and not a dictator, and then exactly one of
$D_0(f),D_1(f)$ equals $1$ and the other equals $0$.
\end{enumerate}
The three cases are therefore distinguished by the pair
$(D_0(f),D_1(f))$ alone.
\end{theorem}

\begin{proof}
By the last sentence of Lemma~\ref{lem:endpoint} a dictator has
$D_0=D_1=1$, and is in particular degenerate.  The three cases are
therefore exhaustive and pairwise disjoint: (i) is the negation of
degeneracy, and (ii) and~(iii) split the degenerate $f$ according to
whether $f$ is a dictator.  This also gives the values of $(D_0,D_1)$
in case~(ii).  In case~(iii), $f$ is degenerate, so $D_0=1$ or $D_1=1$;
and at most one of $D_0,D_1$ is non-zero by
Lemma~\ref{lem:endpoint}(iii), $f$ not being a dictator.  Hence one of
the two equals $1$ and the other equals $0$.

For the last sentence, the three conditions on $(D_0,D_1)$ just
obtained, namely $D_0\ne1$ and $D_1\ne1$; $D_0=D_1=1$; and one of them
equal to $1$ and the other to $0$, are mutually exclusive.  Since the
three cases are exhaustive, the pair $(D_0(f),D_1(f))$ determines which
of them holds.
\end{proof}

\begin{table}[t]
\centering\small
\begin{tabular}{@{}llll@{}}
\toprule
regime & $(D_0,D_1)$ & $\E[T]$ & reference \\
\midrule
non-degenerate & $D_0\ne1$, $D_1\ne1$ & $\Theta\bigl(N(1+\log d_f(s))\bigr)$
  & Theorem~\ref{thm:bin-nondict} \\
dictator & $(1,1)$ & $\Theta\bigl(N^2\Ent(p_0)\bigr)$
  & Theorem~\ref{thm:bin-dict}, Corollary~\ref{cor:dict-asymp} \\
degenerate non-dictator & $(1,0)$ or $(0,1)$ & $\Theta\bigl(N^{2-1/m}\bigr)$
  & Theorem~\ref{thm:bin-deg} \\
\bottomrule
\end{tabular}
\caption{The three regimes of Theorem~\ref{thm:trichotomy}, the
condition on the endpoint degrees that defines each, and the expected
consensus times proved for them.  In the first two rows the
estimate is uniform in the initial state $S_0=s$, with $p_0=s/N$; in the
third it is for $\max_{1\le s\le N-1}\E[T\mid S_0=s]$, and $m\ge2$ is
the least size of a minterm of the residual rule of $f$, or of its dual
when $D_1(f)=1$.}
\label{tab:trichotomy}
\end{table}

\begin{definition}[Attracting endpoints and $d_f$]\label{def:attracting}
Let $f$ be monotone, non-constant and not a dictator.  An absorbing
endpoint is \emph{attracting} when the drift points towards it
in a neighbourhood of it.  By Lemma~\ref{lem:endpoint}, $0$ is attracting
if and only if $D_0=0$ and $1$ is attracting if and only if $D_1=0$; at
least one endpoint is attracting.  Write $d_f(s)$ for the number of
agents that must change state before the nearest attracting consensus is
reached, that is,
\[
d_f(s)\;=\;
\begin{cases}
\min(s,N-s) & \text{if }D_0(f)=D_1(f)=0\quad\text{(case (c))},\\
N-s & \text{if }D_0(f)\ge 1\quad\text{(case (a))},\\
s & \text{if }D_1(f)\ge 1\quad\text{(case (b))}.
\end{cases}
\]
\end{definition}

\begin{lemma}[Structure of a degenerate rule]\label{lem:degstruct}
Let $f$ be monotone, non-constant, not a dictator, with $D_0(f)=1$, say
$f(e_j)=1$.  Then
\[
f(x)\;=\;x_j\vee\tilde f(x_{-j}),\qquad \tilde f(x_{-j}):=f(x^{j\to0}),
\]
where the residual rule $\tilde f$ is monotone, $\tilde f\not\equiv 0$,
and every minterm of $\tilde f$ has size at least $2$ (residual rules and
minterms are as in Section~\ref{sec:intro}).  Moreover $D_1(f)=0$.  The
dual statement holds when $D_1(f)=1$.
\end{lemma}

\begin{proof}
By Lemma~\ref{lem:endpoint}(i), $f\ge x_j$, so $f(x)=1$ whenever $x_j=1$,
and $f(x)=f(x^{j\to0})=\tilde f(x_{-j})$ when $x_j=0$; this is the displayed
formula, and $\tilde f$ is monotone as a restriction of $f$.  If
$\tilde f\equiv0$ then $f=x_j$, a dictator.  If $\tilde f$ had a minterm
$\{i\}$ of size $1$ then $f(e_i)=\tilde f(e_i|_{-j})=1$ with $i\ne j$,
contradicting $D_0(f)=1$.  Finally, for $i\neq j$ the vector
$\mathbf 1-e_i$ has $j$th coordinate $1$, so $f(\mathbf 1-e_i)=1$; and
$f(\mathbf 1-e_j)=\tilde f(\mathbf 1)=1$ because $\tilde f$ is monotone and
not identically $0$.  Hence $D_1(f)=0$.
\end{proof}

\subsection{The birth-death reduction}

\begin{proposition}[The count is a birth-death chain]\label{prop:bd}
The count $(S_t)_{t\ge0}$ is a Markov chain on
$\{0,1,\dots,N\}$, with $0$ and $N$ absorbing and, for $1\le i\le N-1$ and
$\pi=i/N$,
\[
\beta_i\;:=\;\Pp(S_{t+1}=i+1\mid S_t=i)\;=\;(1-\pi)g(\pi),\qquad
\delta_i\;:=\;\Pp(S_{t+1}=i-1\mid S_t=i)\;=\;\pi\bigl(1-g(\pi)\bigr).
\]
In particular $\beta_i-\delta_i=h(\pi)$.
\end{proposition}

\begin{proof}
The increment is $S_{t+1}-S_t=X_V^{(t+1)}-X_V^{(t)}$, which is $+1$ exactly
when $X_V^{(t)}=0$ and $f(X^{(t)}_{J_1},\dots,X^{(t)}_{J_r})=1$, and $-1$
exactly when $X_V^{(t)}=1$ and the new value is $0$.  Given $\mathcal F_t$
these two events are independent, since $V$ is independent of
$(J_1,\dots,J_r)$, and they have probabilities $1-\pi$ and $g(\pi)$ with
$\pi=S_t/N$.  Both resulting probabilities are functions of $S_t$ alone, so
$(S_t)$ is Markov; at $S_t\in\{0,N\}$ they vanish because $g(0)=0$ and
$g(1)=1$.  Finally
$\beta_i-\delta_i=(1-\pi)g(\pi)-\pi(1-g(\pi))=g(\pi)-\pi=h(\pi)$.
\end{proof}

\begin{remark}[Drift identity]\label{rem:drift}
Taking expectations in Proposition~\ref{prop:bd} gives
$\E[S_{t+1}-S_t\mid\mathcal F_t]=\beta_{S_t}-\delta_{S_t}=h(p_t)$, which
identifies $h=g-\mathrm{id}$ as the drift of the number of ones and
explains the terminology.  It is not used in any proof below: the
estimates all run through the step law of Proposition~\ref{prop:bd}
instead.
\end{remark}

Thus $T$ is the absorption time of an explicit birth-death chain, and
the scale function gives its expectation in closed form.  Define
\[
\gamma_0=1,\qquad \gamma_i=\prod_{l=1}^{i}\frac{\delta_l}{\beta_l}\ \ (1\le i\le N-1),
\qquad \Sigma_i=\sum_{l<i}\gamma_l\ \ (0\le i\le N),
\]
so $\Sigma_0=0$, $\Sigma_1=1$, and
\begin{equation}\label{eq:pgqg}
\beta_j\gamma_j\;=\;\delta_j\gamma_{j-1}\qquad(1\le j\le N-1).
\end{equation}

Here $\Sigma$ is the \emph{scale function} of the chain and
$w_j=1/(\beta_j\gamma_j)$ its \emph{speed measure}, the discrete
counterparts of the scale and the speed of a one-dimensional diffusion;
see Karlin and McGregor~\cite{KM57} for the original construction, or
Karlin and Taylor~\cite{KT75} and Levin and
Peres~\cite[Section~2.5]{LP17} for textbook treatments.  Two consequences of~\eqref{eq:pgqg}
explain the names and are used below.  First, \eqref{eq:pgqg} says
exactly that $\Sigma$ is harmonic for the chain on $\{1,\dots,N-1\}$,
that is,
$\beta_s\bigl(\Sigma_{s+1}-\Sigma_s\bigr)=\delta_s\bigl(\Sigma_s-\Sigma_{s-1}\bigr)$,
so $\Sigma_{S_{t\wedge T}}$ is a bounded martingale and optional stopping
gives the gambler's-ruin identity
\[
\Pp\bigl(S_T=N\mid S_0=s\bigr)\;=\;\frac{\Sigma_s}{\Sigma_N}:
\]
in the coordinate $\Sigma$ the chain is a martingale, which is what
putting it \emph{on its natural scale} means.  Second, \eqref{eq:pgqg}
says that $w$ is a reversible measure for the chain,
$w_j\beta_j=w_{j+1}\delta_{j+1}$, so $w_j$ measures how long the chain
lingers at $j$.

The next proposition is the classical formula for the expected
absorption time of a birth-death chain in terms of its scale function
and speed measure \cite{KT75,LP17}.  The quantity $G(s,j)$ appearing in
it is the \emph{Green's function} of the chain killed at $\{0,N\}$: it
is the expected number of time steps spent at $j$ before absorption,
starting from $s$, so that summing it over $j$ gives the expected
absorption time.  In the scale-and-speed form the formula reads
$G(s,j)=\Sigma_{s\wedge j}(\Sigma_N-\Sigma_{s\vee j})\Sigma_N^{-1}w_j$,
a harmonic factor in $s$ and $j$ times the speed measure at $j$.  We
include the short proof because it is the explicit form of $G$, and not
merely its existence, that every estimate below uses.

\begin{proposition}[Green's function]\label{prop:green}
Let $\varphi(s)=\E[T\mid S_0=s]$ for $0\le s\le N$.  Then
$\varphi(0)=\varphi(N)=0$ and, for $1\le s\le N-1$,
\[
\varphi(s)\;=\;\sum_{j=1}^{N-1}G(s,j),\qquad
G(s,j)\;=\;\frac{\Sigma_{s\wedge j}\,(\Sigma_N-\Sigma_{s\vee j})}{\Sigma_N}
\cdot\frac{1}{\beta_j\gamma_j}.
\]
\end{proposition}

\begin{proof}
All $\beta_j,\delta_j$ are strictly positive for $1\le j\le N-1$, so from
any state the chain reaches $0$ within $N$ steps with probability at least
$\prod_{j=1}^{N-1}\delta_j>0$; hence $\E[T]<\infty$.  Conditioning on the
first step, $\varphi$ solves the Poisson equation
$\beta_s(\varphi(s+1)-\varphi(s))-\delta_s(\varphi(s)-\varphi(s-1))=-1$ for
$1\le s\le N-1$ with $\varphi(0)=\varphi(N)=0$.  This system has at most
one solution, by the discrete maximum principle: a solution $u$ of the
homogeneous system satisfies
$(\beta_s+\delta_s)u(s)=\beta_su(s+1)+\delta_su(s-1)$, so $u(s)$ is a
convex combination of $u(s-1)$ and $u(s+1)$ with strictly positive
weights.  If $u$ attained its maximum over $\{0,\dots,N\}$ at an interior
point $s$, then $u(s-1)=u(s)=u(s+1)$, and propagating this to the
boundary gives $\max u=u(0)=0$.  The same applies to $-u$, so $u\equiv0$.
Write
$\nabla(s)=\varphi(s)-\varphi(s-1)$.  The recursion is
$\nabla(s+1)=(\delta_s/\beta_s)\nabla(s)-1/\beta_s$; dividing by $\gamma_s$
and using~\eqref{eq:pgqg},
\[
\frac{\nabla(s+1)}{\gamma_s}=\frac{\nabla(s)}{\gamma_{s-1}}-\frac{1}{\beta_s\gamma_s},
\qquad\text{so}\qquad
\nabla(i+1)=\gamma_i\Bigl(\nabla(1)-\sum_{l=1}^{i}\frac{1}{\beta_l\gamma_l}\Bigr).
\]
Summing, and exchanging the order of summation in the double sum,
\[
\varphi(s)=\sum_{i=0}^{s-1}\nabla(i+1)
=\nabla(1)\Sigma_s-\sum_{l=1}^{s-1}\frac{\Sigma_s-\Sigma_l}{\beta_l\gamma_l}.
\]
Imposing $\varphi(N)=0$ gives
$\nabla(1)=\Sigma_N^{-1}\sum_{l=1}^{N-1}(\Sigma_N-\Sigma_l)/(\beta_l\gamma_l)$.
Substituting, the coefficient of $1/(\beta_j\gamma_j)$ is
$\Sigma_s(\Sigma_N-\Sigma_j)/\Sigma_N$ when $j\ge s$ and
$\bigl[\Sigma_s(\Sigma_N-\Sigma_j)/\Sigma_N\bigr]-(\Sigma_s-\Sigma_j)
=\Sigma_j(\Sigma_N-\Sigma_s)/\Sigma_N$ when $j\le s$.  These are the two
branches of the stated formula.
\end{proof}

\begin{remark}[Dictator sanity check]\label{rem:dictcheck}
For the dictator $x_j$ one has $g=\mathrm{id}$, so
$\beta_i=\delta_i=i(N-i)/N^2$, $\gamma\equiv1$ and $\Sigma_i=i$, and
Proposition~\ref{prop:green} becomes the discrete-Laplacian Green's
function computation carried out in Section~\ref{sec:dict}.
\end{remark}

Using~\eqref{eq:pgqg} we record two rewritings of $G$.  Set
\[
\Lambda_j=\frac{\Sigma_j}{\gamma_{j-1}}\;\ge\;1,\qquad
M_j=\frac{\Sigma_N-\Sigma_j}{\gamma_j}\;\ge\;1\qquad(1\le j\le N-1),
\]
the two inequalities holding because $\Sigma_j\ge\gamma_{j-1}$ and
$\Sigma_N-\Sigma_j\ge\gamma_j$.  Then for $j\le s$,
\begin{equation}\label{eq:Gleft}
G(s,j)\;=\;\Bigl(1-\frac{\Sigma_s}{\Sigma_N}\Bigr)\frac{\Lambda_j}{\delta_j},
\end{equation}
for $j\ge s$,
\begin{equation}\label{eq:Gright}
G(s,j)\;=\;\frac{\Sigma_s}{\Sigma_j}\,\widehat G(j),
\qquad
G(s,j)\;\le\;\frac{\Sigma_s}{\gamma_{j-1}}\cdot\frac{1}{\delta_j},
\end{equation}
and for all $s$ and $j$,
\begin{equation}\label{eq:Ghat}
G(s,j)\;\le\;\widehat G(j)\;:=\;G(j,j)\;=\;
\frac{\Sigma_j(\Sigma_N-\Sigma_j)}{\Sigma_N\,\beta_j\gamma_j}
\;\le\;\min\Bigl(\frac{\Lambda_j}{\delta_j},\ \frac{M_j}{\beta_j}\Bigr),
\end{equation}
because $\Sigma_{s\wedge j}\le\Sigma_j$,
$\Sigma_N-\Sigma_{s\vee j}\le\Sigma_N-\Sigma_j$, with equality in both
when $s=j$ (giving $\widehat G(j)=G(j,j)$), and
$\Sigma_j(\Sigma_N-\Sigma_j)/\Sigma_N\le\min(\Sigma_j,\Sigma_N-\Sigma_j)$.
In particular $\varphi(s)\le\sum_{j=1}^{N-1}\widehat G(j)$ for every $s$.

\subsection{Notation for the estimates}

Write $\pi=j/N$ and
\[
\beta(\pi)=(1-\pi)g(\pi),\qquad \delta(\pi)=\pi(1-g(\pi)),\qquad
R(\pi)=\frac{\delta(\pi)}{\beta(\pi)}=\frac{\pi(1-g(\pi))}{(1-\pi)g(\pi)},
\]
so that $\beta_j=\beta(j/N)$, $\delta_j=\delta(j/N)$ and
$\gamma_i=\prod_{l\le i}R(l/N)$.
Since $R-1=-h/\beta$, the sign of $R-1$ is opposite to the sign of $h$.
For two positive functions we write $u\asymp v$ when $c\,v\le u\le C\,v$ for
constants $0<c\le C<\infty$ depending only on $f$.  Harmonic numbers are
\[
H_n=\sum_{l=1}^{n}\frac1l\ \ (n\ge1),\qquad H_0=0 .
\]
Whenever $g$ has a unique interior fixed point $p_*$ and a half-width
$\eta\in\bigl(0,\min(p_*,1-p_*)/2\bigr)$ has been named, we write
\begin{equation}\label{eq:window}
i_0=\lfloor p_*N\rfloor,\qquad
L=\lfloor(p_*-\eta)N\rfloor,\qquad
U=\lceil(p_*+\eta)N\rceil
\end{equation}
for the centre and the two endpoints of the \emph{window} of half-width
$\eta$ at $p_*$.  The symbols $i_0$, $L$ and $U$ carry no other meaning
anywhere below; each result that uses them names the half-width it forms
them from.

The estimates are organised so that each group of symbols is settled in one
place.  Lemma~\ref{lem:harmonic} owns the elementary sums that both theorem
proofs need repeatedly; Lemma~\ref{lem:rates} owns the rates $\beta$ and
$\delta$; Lemmas~\ref{lem:scale} and~\ref{lem:degprofile} own the scale
function, for a non-degenerate and for a degenerate rule respectively;
Lemma~\ref{lem:duality} owns the reflection $f\mapsto f^*$.  The two
theorems combine their outputs through the Green's-function
identities~\eqref{eq:Gleft},~\eqref{eq:Gright} and~\eqref{eq:Ghat}.
Table~\ref{tab:notation} collects the running notation.

\begin{table}[!t]
\centering\small
\begin{tabular}{@{}lll@{}}
\toprule
symbol & what it is & where it is fixed \\
\midrule
$u\asymp v$ & $c\,v\le u\le C\,v$ with $c,C$ depending only on $f$
  & this subsection \\
$D_0$, $D_1$ & the endpoint degrees $g'(0)$ and $g'(1)$
  & \eqref{eq:endpointdeg} \\
$\varphi(s)$ & the expected consensus time $\E[T\mid S_0=s]$
  & Proposition~\ref{prop:green} \\
$d_f(s)$ & steps to the nearest attracting consensus
  & Definition~\ref{def:attracting} \\
$\beta_j$, $\delta_j$ & birth and death probabilities at $j$
  & Proposition~\ref{prop:bd} \\
$R(\pi)$ & the ratio $\delta(\pi)/\beta(\pi)$ & this subsection \\
$\gamma_j$ & the product $\prod_{l\le j}R(l/N)$ & before~\eqref{eq:pgqg} \\
$\Sigma_j$ & the scale function $\sum_{l<j}\gamma_l$ & before~\eqref{eq:pgqg} \\
$\Lambda_j$ & $\Sigma_j/\gamma_{j-1}\ge1$ & before~\eqref{eq:Gleft} \\
$\Psi$ & the potential $\int_0^\pi\log(\beta/\delta)$
  & Section~\ref{sec:potential} \\
$\mathcal W_j$ & the plateau width $\min(\Lambda_j,M_j)\ge1$
  & Section~\ref{sec:potential} \\
$\rho_s(j)$ & the damping factor of~\eqref{eq:decomp}, $\le1$
  & Section~\ref{sec:potential} \\
$M_j$ & $(\Sigma_N-\Sigma_j)/\gamma_j\ge1$ & before~\eqref{eq:Gleft} \\
$G(s,j)$ & Green's function of the chain killed at $\{0,N\}$
  & Proposition~\ref{prop:green} \\
$\widehat G(j)$ & an envelope for $G(\cdot,j)$: $G(s,j)\le\widehat G(j)$
  for every $s$ & \eqref{eq:Ghat} \\
$m_0$, $m_1$ & least minterm sizes of $f$ and of the dual $f^*$
  & Lemma~\ref{lem:rates} \\
$p_*$ & the interior fixed point of $g$, when there is one
  & Theorem~\ref{thm:fp} \\
$\eta$ & half-width of the window at $p_*$ & \eqref{eq:window} \\
$i_0$, $L$, $U$ & centre and the two endpoints of that window
  & \eqref{eq:window} \\
$\nu$ & a lower bound $\nu>1$ for $R$: on $(0,1)$ in~(b),
  left of the window in~(c) & Lemma~\ref{lem:scale} \\
$m$ & least minterm size of the residual rule $\tilde f$
  & Lemma~\ref{lem:degstruct} \\
\bottomrule
\end{tabular}
\caption{The running notation of this section, and where each item is
fixed.  The table is a reference for the whole of
Section~\ref{sec:binary}, not only for the subsection that contains it.
The further quantities $K$, $\pi_0$ and $c_0$ of the degenerate case are
fixed in Lemma~\ref{lem:degprofile}; of these, $K=N^{(m-1)/m}$ depends on
$N$ as well as on $f$, while $\pi_0$ and $c_0$ depend on $f$ alone.}
\label{tab:notation}
\end{table}

Throughout the rest of Section~\ref{sec:binary}, $c$ and $C$ denote positive
finite constants depending only on $f$, whose value may change from one
occurrence to the next; a constant that is referred to again later carries a
name.  No such constant is ever allowed to depend on $N$: a bound that holds
only for large $N$ is stated with the explicit hypothesis $N\ge N_0(f)$, where
$N_0(f)$ is a threshold depending only on $f$, and the finitely many smaller
$N$ are absorbed into $c$ and $C$ at the point where the bound is applied.
When a result introduces a parameter of its own, its threshold may depend on
that parameter as well and is then written out in full, as $N_0(f,\eta)$;
the bare $c$ and $C$ of its conclusion may not depend on the parameter, and a
constant that does carries a name recording the dependence, as $A_\nu$ does
in Lemma~\ref{lem:harmonic}(c).

We record \emph{absorbing the small $N$} once.  To prove a two-sided bound
between two finite positive quantities it suffices to prove it for
$N\ge N_0(f)$: the finitely many smaller $N$ contribute finitely many finite
positive values of the ratio, which enlarging $C$ and shrinking $c$ covers.
Every lemma whose conclusion needs $N\ge N_0(f)$ states it.

\subsection{The potential picture}
\label{sec:potential}

This subsection is a guide to the trichotomy of
Theorem~\ref{thm:trichotomy} and to Lemmas~\ref{lem:rates},~\ref{lem:scale}
and~\ref{lem:degprofile}, not a formal ingredient in their proofs.
We trace, step by step, how $\varphi(s)=\E[T\mid S_0=s]$, expressed as
a sum of Green's-function values by Proposition~\ref{prop:green}, is
controlled by a single scalar function $\Psi$, and how the local behaviour
of $\Psi$ near its boundary zeros determines the three consensus-time regimes.

\paragraph{Step 1: Reducing to the envelope.}
Proposition~\ref{prop:green} gives $\varphi(s)=\sum_{j=1}^{N-1}G(s,j)$,
where (from~\eqref{eq:Ghat}) each $G(s,j)$ is bounded above by the
\emph{envelope}
\[
\widehat G(j)\;=\;\frac{\Sigma_j(\Sigma_N-\Sigma_j)}{\Sigma_N\,\beta_j\gamma_j}.
\]
Dividing $G(s,j)$ by $\widehat G(j)$ and reading off the ratio from
Proposition~\ref{prop:green} yields the decomposition
\begin{equation}\label{eq:decomp}
\varphi(s)\;=\;\sum_{j=1}^{N-1}\rho_s(j)\,\widehat G(j),
\qquad
\rho_s(j)\;=\;
\begin{cases}
\dfrac{\Sigma_N-\Sigma_s}{\Sigma_N-\Sigma_j}, & j\le s,\\[8pt]
\dfrac{\Sigma_s}{\Sigma_j}, & j\ge s,
\end{cases}
\end{equation}
where $0\le\rho_s(j)\le1$ and $\rho_s(s)=1$.  By a standard gambler's-ruin
calculation on the birth--death chain, $\rho_s(j)$ equals the probability
that the chain started at~$s$ visits~$j$ before being absorbed at the far
boundary (at $N$ if $j<s$, at $0$ if $j>s$).  The decomposition splits the
problem: $\widehat G(j)$ is an $s$-independent envelope, while $\rho_s$
records how far from~$s$ the chain can reach.

The dominant factor in $\widehat G$ is $\gamma_j$.
From~\eqref{eq:pgqg}, $\beta_j\gamma_j=\delta_j\gamma_{j-1}$, so the denominator
factors as
\begin{equation}\label{eq:parallel}
\frac{1}{\widehat G(j)}\;=\;
\frac{\Sigma_N\,\delta_j\gamma_{j-1}}{\Sigma_j(\Sigma_N-\Sigma_j)}
\;=\;\delta_j\gamma_{j-1}\!\left(\frac{1}{\Sigma_j}+\frac{1}{\Sigma_N-\Sigma_j}\right),
\end{equation}
and therefore
\begin{equation}\label{eq:Gmin}
\widehat G(j)\;\asymp\;\frac{\min(\Sigma_j,\,\Sigma_N-\Sigma_j)}{\delta_j\gamma_{j-1}}.
\end{equation}
The identity~\eqref{eq:pgqg} is what makes both summands in~\eqref{eq:parallel}
share the single factor $\delta_j\gamma_{j-1}$, and hence why $\gamma$ is the
key object: $\widehat G(j)$ is large precisely where $\gamma_{j-1}$ is small,
and $\Sigma_j=\sum_{l<j}\gamma_l$ and $\Sigma_N-\Sigma_j=\sum_{l\ge j}\gamma_l$
are themselves sums of $\gamma$-values whose size we need to understand.

\paragraph{Step 2: The potential function.}
Recall $\gamma_j=\prod_{l=1}^j R(l/N)$ where $R(\pi)=\delta(\pi)/\beta(\pi)$.
Since $\log\gamma_j=\sum_{l=1}^j\log R(l/N)$ is $N$ times a Riemann sum for
$\int_0^{j/N}\log R(u)\,du$, we define the \emph{potential}
\[
\Psi(\pi)\;=\;\int_0^{\pi}\log\frac{\beta(u)}{\delta(u)}\,du
\;=\;-\int_0^\pi\log R(u)\,du,
\]
finite on $[0,1]$ (the integrand has at worst a logarithmic singularity at
each endpoint).  Then
\begin{equation}\label{eq:gammaPsi}
\gamma_j\;=\;\exp\bigl(-N\,\Psi(j/N)+O(\log N)\bigr).
\end{equation}
The $O(\log N)$ error arises from endpoint contributions to the Euler--Maclaurin
formula; the proofs of Lemmas~\ref{lem:scale} and~\ref{lem:degprofile}
work directly with $\gamma_j$ to avoid it.

\paragraph{Step 3: What $\Psi$ encodes.}
Three facts determine the shape of $\Psi$ completely.

\emph{Sign.}  Since $\Psi'(\pi)=\log(\beta(\pi)/\delta(\pi))$,
we have $\operatorname{sgn}\Psi'=\operatorname{sgn}\,h$ where $h=\beta-\delta$.
By Corollary~\ref{cor:sign}, $h$ changes sign \emph{at most once} in $(0,1)$:
where $h<0$, $\Psi$ decreases; where $h>0$, $\Psi$ increases.
If $h$ does not change sign, $\Psi$ is monotone.
If $h$ has an interior zero $p_*$ (so $\beta(p_*)=\delta(p_*)$), then $\Psi$
has a critical point there; the Bernoulli--Poincar\'e inequality
(Theorem~\ref{thm:fp}) gives $\Psi''(p_*)=h'(p_*)/\beta(p_*)>0$, so it is a
non-degenerate \emph{minimum} of $\Psi$.

\emph{Endpoint slopes.}  As $\pi\downarrow0$, $\delta(\pi)\sim\pi$ and
$\beta(\pi)/\delta(\pi)\to D_0$, so $\Psi'(0)=\log D_0$
(convention: $\log 0=-\infty$).  Symmetrically $\Psi'(1)=-\log D_1$, giving
\begin{equation}\label{eq:PsiEnds}
\Psi'(0)\;=\;\log D_0,\qquad \Psi'(1)\;=\;-\log D_1.
\end{equation}
Thus $\Psi'$ has a zero at an endpoint precisely when the endpoint degree
equals~$1$; Theorem~\ref{thm:trichotomy} is a count of these boundary zeros:
none in regime~(i), one in regime~(iii), both in regime~(ii).

\emph{Interior curvature.}  At $p_*$, $\beta(p_*)=\delta(p_*)$ so
$\Psi'(p_*)=0$, and the Bernoulli Poincar\'e inequality (Theorem~\ref{thm:fp})
gives $\Psi''(p_*)=h'(p_*)/\beta(p_*)>0$: the interior critical point is
always a non-degenerate minimum of $\Psi$.

\paragraph{Step 4: Approximating $\Sigma_j$ and $\Sigma_N-\Sigma_j$.}
Both are sums of $\gamma$-values.  Dividing by the reference value $\gamma_{j-1}$
(resp.\ $\gamma_j$) and using~\eqref{eq:gammaPsi} with $\pi=j/N$ gives
\begin{equation}\label{eq:LamInt}
\frac{\Sigma_j}{\gamma_{j-1}}
\;\approx\;\sum_{l=0}^{j-1}e^{-N(\Psi(l/N)-\Psi(\pi))}
\;\approx\;N\!\int_{0}^{\pi}\!e^{-N(\Psi(v)-\Psi(\pi))}\,dv,
\qquad
\frac{\Sigma_N-\Sigma_j}{\gamma_j}
\;\approx\;N\!\int_{\pi}^{1}\!e^{-N(\Psi(v)-\Psi(\pi))}\,dv.
\end{equation}
Consecutive terms in these sums satisfy $\text{term}_{l+1}/\text{term}_l
\approx e^{-\Psi'(\pi)}$.  The dominant term of $\Sigma_j/\gamma_{j-1}$ is the
\emph{rightmost} ($l=j-1$, value~$\approx1$): when $\Psi'(\pi)>0$ the terms
decrease as $l$ falls from $j-1$, so the sum is geometric and
$\Sigma_j/\gamma_{j-1}\asymp1$.  The dominant term of $(\Sigma_N-\Sigma_j)/\gamma_j$
is the \emph{leftmost} ($l=j$, value~$\approx1$): when $\Psi'(\pi)<0$ that
sum is geometric and $(\Sigma_N-\Sigma_j)/\gamma_j\asymp1$.  In summary,
when $|\Psi'(\pi)|\ge c>0$,
\begin{equation}\label{eq:noplateau}
\frac{\Sigma_j}{\gamma_{j-1}}\asymp1\ \text{ where }\Psi'(\pi)\ge c,
\qquad
\frac{\Sigma_N-\Sigma_j}{\gamma_j}\asymp1\ \text{ where }\Psi'(\pi)\le-c.
\end{equation}

When $\Psi'(\pi)\approx0$ the terms vary slowly and the integral
approximation in~\eqref{eq:LamInt} applies.  Whichever of the two intervals
$[0,\pi]$ or $[\pi,1]$ is on the side \emph{away from} the (nearest) minimum
of $\Psi$ has $\Psi(v)\ge\Psi(\pi)$ throughout: on that side the integrand
is at most~$1$ and is $\asymp1$ only in the \emph{slab} $\{v:\Psi(v)\le\Psi(\pi)+\frac1N\}$.
Outside the slab the integrand is at most $e^{-1}$, so the integral is
$\asymp N$ times the slab length (the factor $N$ converts slab measure to a
Riemann-sum count of lattice points), giving
\begin{equation}\label{eq:plateau}
\frac{\Sigma_j}{\gamma_{j-1}}\;\approx\;N\cdot\bigl|\{v<\pi:\Psi(v)\le\Psi(\pi)+\tfrac1N\}\bigr|
\quad\text{(when $[0,\pi]$ is the side away from the minimum),}
\end{equation}
symmetrically for $(\Sigma_N-\Sigma_j)/\gamma_j$ when $[\pi,1]$ is that side.
On the opposite interval $\Psi$ dips below $\Psi(\pi)$, so that integral exceeds
the slab and is the \emph{larger} of the two.  Define the \emph{plateau width}
\[
\mathcal W_j\;:=\;\min\!\left(\frac{\Sigma_j}{\gamma_{j-1}},\,
  \frac{\Sigma_N-\Sigma_j}{\gamma_j}\right)\;\ge\;1
\]
(each ratio is $\ge1$ since $\Sigma_j\ge\gamma_{j-1}$ and $\Sigma_N-\Sigma_j\ge\gamma_j$);
it equals the slab count on the side away from the minimum.

\paragraph{Step 5: The master estimate.}
From~\eqref{eq:Gmin} and the definition of $\mathcal W_j$, two cases arise:
\begin{itemize}
\item \emph{$\Sigma_j/\gamma_{j-1}\le(\Sigma_N-\Sigma_j)/\gamma_j$
  (left branch active):}
$\widehat G(j)\asymp\Sigma_j/(\delta_j\gamma_{j-1})
=(\Sigma_j/\gamma_{j-1})/\delta_j=\mathcal W_j/\delta_j$.
Near $j=0$, $\delta_j\asymp j/N$, so $\widehat G(j)\asymp N\mathcal W_j/j$.
\item \emph{$(\Sigma_N-\Sigma_j)/\gamma_j\le\Sigma_j/\gamma_{j-1}$
  (right branch active):}
$\widehat G(j)\asymp(\Sigma_N-\Sigma_j)/(\delta_j\gamma_{j-1})
=(\Sigma_N-\Sigma_j)/(\beta_j\gamma_j)=\mathcal W_j/\beta_j$,
using $\delta_j\gamma_{j-1}=\beta_j\gamma_j$.
The right branch is active near $j=0$ only when $\Psi'(0)\ge0$
(i.e.\ $D_0\ge1$, by~\eqref{eq:PsiEnds}); then $\beta_j\asymp D_0\cdot j/N\asymp j/N$,
so again $\widehat G(j)\asymp N\mathcal W_j/j$.
\end{itemize}
In both cases, and symmetrically near $j=N$,
\begin{equation}\label{eq:master}
\widehat G(j)\;\asymp\;\frac{N\,\mathcal W_j}{\min(j,N-j)}.
\end{equation}
Substituting into~\eqref{eq:decomp},
\[
\varphi(s)\;\approx\;
N\int_0^1\frac{\rho_s(\pi)\,\mathcal W(\pi)}{\min(\pi,1-\pi)}\,d\pi:
\]
the expected absorption time is $N$ times the integral of the plateau
width against the harmonic weight $1/\min(\pi,1-\pi)$.

Whether the outcome is a logarithm or a power of $N$ turns on the dyadic
structure of $\widehat G$: when $\mathcal W_j\asymp1$ the blocks
$[2^k,2^{k+1})$ all contribute $\asymp N$, summing to $N\log N$ over
$O(\log N)$ comparable blocks; when $\mathcal W_j$ grows or decays as a
power of $j$ (or $N/j$), the blocks are geometric and a single scale dominates.

\paragraph{Step 6: The four cases.}
By~\eqref{eq:PsiEnds}, $\Psi'$ vanishes at an endpoint precisely when the
endpoint degree equals~$1$.  The only other available zero is the interior
$p_*$ (when $D_0=D_1=0$, Lemma~\ref{lem:endpoint}(iii)).  This yields four
cases.

\begin{itemize}[itemsep=4pt]
\item \emph{No zero of $\Psi'$ in $[0,1]$
  (cases~(a),(b) of Corollary~\ref{cor:sign}).}
By~\eqref{eq:noplateau}, $\mathcal W_j\asymp1$ everywhere, so
$\widehat G(j)\asymp N/\min(j,N-j)$.  The dyadic blocks $[2^k,2^{k+1})$
and their mirrors each contribute $\asymp N$; the blocks are comparable, and
the weight $\rho_s\asymp1$ towards the attracting endpoint selects the
$\asymp1+\log d_f(s)$ blocks in the range $[1,d_f(s)]$.
Total: $\varphi(s)\asymp N(1+\log d_f(s))$, i.e.\ regime~(i).

\item \emph{Interior zero $\Psi'(p_*)=0$, $\Psi''(p_*)>0$
  (case~(c) of Corollary~\ref{cor:sign}).}
Since $\Psi(v)-\Psi(p_*)\asymp(v-p_*)^2$, the slab of height $1/N$ has
half-width $\asymp N^{-1/2}$, giving
$\mathcal W_j\asymp\min(\sqrt N,N/|j-i_0|)$ near $i_0=\lfloor p_*N\rfloor$.
With $j\asymp N$ in the window, $\widehat G(j)\asymp\mathcal W_j$, and
summing over dyadic shells of $|j-i_0|$:
\[
\underbrace{\sqrt N\cdot\sqrt N}_{|j-i_0|\le\sqrt N}
+\sum_{\sqrt N<|j-i_0|\le\eta N}\frac{N}{|j-i_0|}
\;\asymp\;N\log N.
\]
The window contributes the same order as the harmonic body, so the total
is still $N\log N$: this is regime~(i).  (The window is only seen by
initial states $s$ inside it, for which $d_f(s)\asymp N$ anyway.)

\item \emph{One boundary zero of $\Psi'$, say $\Psi'(0)=0$ ($D_0=1$).}
As $\pi\downarrow0$, Lemma~\ref{lem:degstruct} gives $\Psi'(\pi)\asymp\pi^{m-1}$
and $\Psi(\pi)\asymp\pi^m$, where $m\ge2$ is the least minterm size of the
residual rule.  The exponential barrier $N\Psi(j/N)\asymp j^m/N^{m-1}$ first
reaches order~$1$ at the \emph{crossover scale} $K=N^{(m-1)/m}$.
From~\eqref{eq:plateau}, on the left half $j\le N/2$:
\[
\mathcal W_j\;\asymp\;
\begin{cases}
j, & j\le K,\\[2pt]
(N/j)^{m-1}, & K\le j\le N/2,
\end{cases}
\qquad
\widehat G(j)\;\asymp\;
\begin{cases}
N, & j\le K,\\[2pt]
(N/j)^{m}, & K\le j\le N/2.
\end{cases}
\]
Below $K$ the barrier $N\Psi(j/N)\ll1$ and the left slab covers all of
$[0,\pi]$, giving $\mathcal W_j=\Sigma_j/\gamma_{j-1}\asymp j$.  Above $K$ the barrier
suppresses $\Sigma_j/\gamma_{j-1}$, the active branch switches to $(\Sigma_N-\Sigma_j)/\gamma_j$, and the slab
width $\asymp1/(N\Psi'(\pi))\asymp(N/j)^{m-1}/N$ gives
$\mathcal W_j=(\Sigma_N-\Sigma_j)/\gamma_j\asymp(N/j)^{m-1}$.  Summing:
\[
\sum_{j\le N/2}\widehat G(j)\;\asymp\;NK+N^m K^{1-m}\;\asymp\;NK
\;=\;N^{2-1/m},
\]
the two ranges contributing equally.  The blocks are geometric on both sides
of $K$ (growing as $Nj$ below, decaying as $N^m j^{-m}$ above), so the
single scale $j\asymp K$ carries the mass.  Since $D_0=1$ forces $h>0$
(Lemma~\ref{lem:endpoint}(i)), the chain is pushed away from~$0$, and the
worst initial state is $s\asymp K$.

\item \emph{Both boundary zeros of $\Psi'$ ($D_0=D_1=1$, dictator).}
Here $\beta\equiv\delta$, so $\Psi\equiv0$: no exponential weighting and
$\mathcal W_j=\min(j,N-j)$.  This is the previous case with $K=N$,
giving $\widehat G(j)\asymp N$ uniformly and
$\sum_j\widehat G(j)\asymp N^2$.  With $\gamma\equiv1$, $\Sigma_j=j$, and
$\rho_s(j)=s/j$ for $j\ge s$: a power-law weight, not a geometric cut-off.
Since $\Psi\equiv0$ the integrals~\eqref{eq:LamInt} can be evaluated exactly:
\[
\varphi(s)\;\approx\;
N^{2}\!\left[(1-p_0)\!\int_0^{p_0}\!\frac{d\pi}{1-\pi}
+p_0\!\int_{p_0}^1\!\frac{d\pi}{\pi}\right]
\;=\;N^{2}\Ent(p_0),
\]
where $\Ent(p)=-p\log p-(1-p)\log(1-p)$, recovering Theorem~\ref{thm:bin-dict}
up to the $O(N)$ of Corollary~\ref{cor:dict-asymp}(a).
\end{itemize}

Table~\ref{tab:plateau} collects the four cases.
\begin{table}[!t]
\centering\small
\begin{tabular}{@{}lllll@{}}
\toprule
zeros of $\Psi'$ & plateau $\mathcal W_j\asymp$ & mass carried at
  & $\sum_j\widehat G(j)$ & regime \\
\midrule
none, $|\Psi'|\ge c$ & $1$ & all scales $j\le N/2$ & $N\log N$ & (i) \\
$p_*$, simple & $\min\bigl(\sqrt N,N/(|j-i_0|+1)\bigr)$
  & all scales of $|j-i_0|$ & $N\log N$ & (i) \\
one endpoint, $\Psi\asymp\pi^m$ & $\min\bigl(j,(N/j)^{m-1}\bigr)$
  & $j\asymp K=N^{(m-1)/m}$ & $N^{2-1/m}$ & (iii) \\
everywhere, $\Psi\equiv0$ & $\min(j,N-j)$ & $j\asymp N$ & $N^{2}$ & (ii) \\
\bottomrule
\end{tabular}
\caption{The four profiles of Section~\ref{sec:potential}, the third row
written for $D_0=1$ and reflected in $j\mapsto N-j$ when $D_1=1$.
The table is the envelope $\widehat G$, which does not see the initial
state; $\rho_s$ then selects the scales that a given~$s$ actually sees.}
\label{tab:plateau}
\end{table}

\begin{remark}[A parallel singularity in the mean-field equation]\label{rem:mfe-singularity}
The third row of Table~\ref{tab:plateau} is a discrete instance of a
phenomenon already identified directly on the mean-field/diffusion
approximation.  Doering, Sargsyan and Sander~\cite{DSS04} solve a
related class of birth-death chains exactly and compare the result
with the naive Fokker--Planck limit of the same master equation; they
find that the two agree only where the deterministic drift is
uniformly small, and that at a boundary where the drift and its
derivative vanish together -- their threshold case, the analogue of
our $\Psi'(0)=0$ -- the exact chain picks up an anomalous power-law
term that the naive continuum equation cannot see.  Hathcock and
Strogatz~\cite{HS22} isolate a one-parameter family of birth-death
chains whose rates vanish at a boundary at a tunable polynomial rate
and obtain a corresponding family of absorption-time universality
classes, indexed by the same kind of exponent that our minterm size
$m$ supplies in Lemma~\ref{lem:degprofile}.  In both cases, as here,
the anomalous exponent is a genuinely discrete phenomenon that a
naive continuum limit does not reproduce on its own.
\end{remark}

\subsection{Estimates on the scale function}

The first lemma collects three elementary sums used repeatedly below.

\begin{lemma}[Harmonic sums]\label{lem:harmonic}
Let $N\ge2$ and let $s$ and $v$ denote integers.
\begin{enumerate}[label=(\alph*),itemsep=2pt]
\item For $s\ge1$,
$\;H_s\ge\max(1,\log s)\ge\frac12\bigl(1+\log s\bigr)$ and
$H_s\le1+\log s$.
\item For $1\le s\le N-1$,
\[
\sum_{1\le j<s}\frac{1}{N-j}\;\le\;\log\frac{N}{N-s}\;\le\;2\bigl(1+\log s\bigr).
\]
\item Let $\nu>1$.  Then
$\displaystyle\sum_{u=1}^{v}\frac{\nu^{-(v-u)}}{u}\;\le\;A_\nu$ for every
$v\ge1$, with $A_\nu<\infty$ depending only on $\nu$.
\end{enumerate}
\end{lemma}

\begin{proof}
(a)  From $H_s\ge H_1=1$ and
$H_s\ge\int_1^{s+1}\frac{dx}{x}=\log(s+1)\ge\log s$ we get
$H_s\ge\max(1,\log s)\ge\frac12(1+\log s)$, the last step because
$\max(1,t)-\frac{1+t}{2}=\frac{|1-t|}{2}\ge0$; and
$H_s\le1+\int_1^{s}\frac{dx}{x}=1+\log s$.

(b)  Substituting $k=N-j$, the sum equals $\sum_{k=N-s+1}^{N-1}\frac1k
\le\int_{N-s}^{N-1}\frac{dx}{x}\le\log\frac{N}{N-s}$, the hypothesis
$s\le N-1$ giving $N-s\ge1$.  For the second inequality: if $s\le N/2$ then
$N-s\ge N/2$ and $\log\frac{N}{N-s}\le\log2\le2\le2(1+\log s)$; if $s>N/2$
then $\log\frac{N}{N-s}\le\log N$ because $N-s\ge1$, while
$\log s\ge\log N-\log2$, so
$2(1+\log s)\ge2+2\log N-2\log2\ge\log N$, the last step because
$\log N\ge0\ge2\log2-2$.

(c)  Split the sum at $u=v/2$.  For $u\le v/2$ we have
$\nu^{-(v-u)}\le\nu^{-v/2}$, so those terms contribute at most
$\nu^{-v/2}H_v$.  For $u>v/2$ we have $1/u<2/v$, so those terms contribute
at most $\frac2v\sum_{n\ge0}\nu^{-n}=\frac2v\cdot\frac{\nu}{\nu-1}$.  Hence
the sum is at most
$\nu^{-v/2}(1+\log v)+\frac{2}{v}\cdot\frac{\nu}{\nu-1}$ by~(a), and both
terms are bounded over $v\ge1$ by a quantity depending only on $\nu$, the
first because $\nu^{-v/2}(1+\log v)\to0$ as $v\to\infty$.
\end{proof}

\begin{lemma}[Sizes of the rates]\label{lem:rates}
Let $f$ be monotone and non-constant, and let $m_0$ be the least size of a
minterm of $f$ and $m_1$ the least size of a minterm of the dual $f^*$.
Then $m_0=1$ if and only if $D_0\ge1$, and $m_1=1$ if and only if
$D_1\ge1$; moreover, uniformly on $[0,1]$,
\[
g(\pi)\asymp\pi^{m_0},\qquad 1-g(\pi)\asymp(1-\pi)^{m_1},\qquad
\beta(\pi)\asymp(1-\pi)\pi^{m_0},\qquad \delta(\pi)\asymp\pi(1-\pi)^{m_1}.
\]
Consequently:
\begin{enumerate}[label=(\alph*),itemsep=2pt]
\item if $D_0\ge1$ then $g(\pi)\ge\pi$ on $[0,1]$, and hence
\[
\frac{1}{\beta_j}\;\le\;\frac{N^2}{j(N-j)}\qquad(1\le j\le N-1);
\]
\item if $m_1=1$, that is if $D_1\ge1$, then
\[
\frac{1}{\delta_j}\;\asymp\;\frac{N^2}{j(N-j)}\qquad(1\le j\le N-1);
\]
\item if $m_0\ge2$ and $m_1\ge2$, so that $D_0=D_1=0$ and $g$ has a unique
interior fixed point $p_*$, then for every half-width $\eta$ with
$0<\eta<\min(p_*,1-p_*)/2$ there is $N_0(f,\eta)$ such that for
$N\ge N_0(f,\eta)$ the window~\eqref{eq:window} of half-width $\eta$
satisfies $1\le L<U\le N-1$ and
\[
\frac{1}{\delta_j}\;\asymp\;\frac{N}{j}\ \ (1\le j\le U),
\qquad
\frac{1}{\beta_j}\;\asymp\;\frac{N}{N-j}\ \ (L\le j\le N-1),
\]
while both $1/\delta_j$ and $1/\beta_j$ are at most $C$ for $L<j<U$.  Here
the threshold $N_0(f,\eta)$ depends on $\eta$, but the constants implicit in
the two relations $\asymp$ and the constant $C$ do not: all three may be
taken to depend on $f$ alone, uniformly over
$\eta\in\bigl(0,\min(p_*,1-p_*)/2\bigr)$.
\end{enumerate}
\end{lemma}

\begin{proof}
Write $g(\pi)=\sum_x f(x)\pi^{|x|}(1-\pi)^{r-|x|}$ and expand in powers of
$\pi$.  The coefficient of $\pi^n$ vanishes for $n<m_0$ and equals the
number of minterms of size $m_0$ for $n=m_0$, so
$g(\pi)=\pi^{m_0}\mathsf{g}_0(\pi)$
with $\mathsf{g}_0$ a polynomial and $\mathsf{g}_0(0)>0$.  Also
$\mathsf{g}_0(\pi)=g(\pi)/\pi^{m_0}>0$
for $\pi\in(0,1]$ because $g$ is strictly increasing with $g(1)=1$.  Hence
$\mathsf{g}_0$ is continuous and positive on the compact interval $[0,1]$ and so
$g(\pi)\asymp\pi^{m_0}$.  Applying the same argument to $f^*$, whose
mean-field map is $1-g(1-\pi)$, gives $1-g(\pi)\asymp(1-\pi)^{m_1}$.  The
statements about $\beta$ and $\delta$ follow by multiplying by $1-\pi$
and $\pi$.
Finally $D_0=g'(0)=\mathsf{g}_0(0)\mathbf 1\{m_0=1\}$, so $D_0\ge1$ if and only if
$m_0=1$, and dually for $D_1$.

(a)  If $D_0\ge1$, say $f(e_i)=1$, then $x\ge e_i$ forces $f(x)\ge f(e_i)=1$,
so $f\ge x_i$ pointwise and $g(\pi)=\E_\pi[f]\ge\E_\pi[x_i]=\pi$ for every
$\pi\in[0,1]$.  Hence $\beta_j=(1-\pi)g(\pi)\ge\pi(1-\pi)=j(N-j)/N^2$.

(b)  With $m_1=1$ the fourth asymptotic reads $\delta(\pi)\asymp\pi(1-\pi)$,
that is $\delta_j\asymp j(N-j)/N^2$.

(c)  By the first sentence, $m_0\ge2$ and $m_1\ge2$ mean $D_0=D_1=0$; in
particular $f$ is not a dictator, a dictator having $D_0=D_1=1$
(Lemma~\ref{lem:endpoint}), so Corollary~\ref{cor:sign} applies and
Lemma~\ref{lem:endpoint}(iii) puts $f$ in its case~(c), which supplies the
interior fixed point $p_*$.  Fix $\eta$ as stated and form the
window~\eqref{eq:window}.  Since
$L\ge(p_*-\eta)N-1\ge\frac{p_*}{2}N-1$ and
$N-U\ge(1-p_*-\eta)N-1\ge\frac{1-p_*}{2}N-1$, there is $N_0(f,\eta)$ with
\[
1\;\le\;\tfrac{p_*}{4}N\;\le\;L\;<\;U\;\le\;N-\tfrac{1-p_*}{4}N\;\le\;N-1
\qquad(N\ge N_0(f,\eta)).
\]
For $1\le j\le U$ we then have $1-j/N\ge1-U/N\ge(1-p_*)/4$, so
$(1-\pi)^{m_1}\asymp1$ and $\delta_j\asymp\pi=j/N$; for $L\le j\le N-1$ we
have $j/N\ge L/N\ge p_*/4$, so $\pi^{m_0}\asymp1$ and
$\beta_j\asymp1-\pi=(N-j)/N$.  Finally, for $L<j<U$ the first of these
gives $1/\delta_j\le CN/j\le CN/L\le 4C/p_*$ and the second gives
$1/\beta_j\le CN/(N-j)\le CN/(N-U)\le 4C/(1-p_*)$.  This proves the claim
that the constants do not depend on $\eta$: every constant produced in this
paragraph is a function of $p_*$, of $m_0$ and $m_1$, and of the constants
in the four asymptotics of the first display, and none of those involves
$\eta$, which is used only through the two inequalities
$\eta<p_*/2$ and $\eta<(1-p_*)/2$ that precede the display.  The threshold
$N_0(f,\eta)$ is the one quantity for which no such claim is made.
\end{proof}

Recall that $f$ is non-degenerate when $D_0\ne1$ and $D_1\ne1$, and that at
most one of $D_0,D_1$ is non-zero unless $f$ is a dictator
(Remark~\ref{rem:d0d1}).  The three cases of the next lemma, $D_0\ge2$ with
$D_1=0$; $D_0=0$ with $D_1\ge2$; and $D_0=D_1=0$, therefore exhaust the
non-degenerate rules.  The pairs they omit are $(D_0,D_1)=(1,1)$, the
dictator, treated in Section~\ref{sec:dict}, and $(D_0,D_1)=(1,0)$ and
$(0,1)$, the degenerate non-dictators, treated in
Theorem~\ref{thm:bin-deg}; by Theorem~\ref{thm:trichotomy} these three
groups account for every monotone non-constant $f$, as Table~\ref{tab:trichotomy}
records.  Its bounds are stated in terms of $\Lambda_j$, $M_j$
and ratios of $\Sigma$; the identity
$\frac{\Sigma_s}{\Sigma_j}\Lambda_j=\frac{\Sigma_s}{\gamma_{j-1}}$ of
parts~(b3) and~(c2) is what feeds the bound
$G(s,j)\le\frac{\Sigma_s}{\gamma_{j-1}}\cdot\frac1{\delta_j}$
of~\eqref{eq:Gright}.

\begin{lemma}[Profile of a non-degenerate rule]\label{lem:scale}
Let $f$ be monotone, non-constant and non-degenerate; then $f$ is not a
dictator, and exactly one of the following three cases occurs.  All constants
below depend only on $f$.
\begin{enumerate}[label=(\alph*),itemsep=3pt]
\item If $D_0\ge2$ and $D_1=0$, then $M_j\le C$ for $1\le j\le N-1$.
\item If $D_0=0$ and $D_1\ge2$, then there is $\nu>1$ such that
\begin{enumerate}[label=(b\arabic*),itemsep=2pt,leftmargin=2.4em]
\item $\Lambda_j\le C$ for $1\le j\le N-1$;
\item $1-\dfrac{\Sigma_s}{\Sigma_N}\ge c$ for $1\le s\le N-1$;
\item $\dfrac{\Sigma_s}{\Sigma_j}\,\Lambda_j=\dfrac{\Sigma_s}{\gamma_{j-1}}
\le C\,\nu^{-(j-s)}$ for $1\le s\le j\le N-1$.
\end{enumerate}
\item If $D_0=D_1=0$, let $p_*$ be the interior fixed point.  There are
$\eta\in\bigl(0,\min(p_*,1-p_*)/2\bigr)$ and $\nu>1$ such that, with
$i_0$, $L$, $U$ the window~\eqref{eq:window} of half-width $\eta$ at $p_*$,
the following hold for every $N\ge N_0(f)$:
\begin{enumerate}[label=(c\arabic*),itemsep=2pt,leftmargin=2.4em]
\item $\Lambda_j\le C$ for $1\le j\le L$, and $M_j\le C$ for $U\le j\le N-1$;
\item $\dfrac{\Sigma_s}{\Sigma_j}\,\Lambda_j=\dfrac{\Sigma_s}{\gamma_{j-1}}
\le C\,\nu^{-(j-s)}$ for $1\le s\le j\le L$;
\item $\dfrac{\Sigma_s}{\Sigma_L}\le C\,\nu^{-(L-s)}$ for $1\le s\le L$;
\item $\dfrac{\Sigma_L}{\Sigma_N}\le\dfrac12$;
\item $L\asymp N$, $N-L\asymp N$, $U\asymp N$ and $N-U\asymp N$; in
particular $1\le L<i_0<U\le N-1$;
\item $\displaystyle\sum_{L<j\le i_0}\Lambda_j+\sum_{i_0\le j<U}M_j
\;\le\;C\,N\log N$.
\end{enumerate}
\end{enumerate}
The restriction $j\le L$ in~(c2) cannot be removed:
past $L$ the ratios $R(l/N)$ approach $1$ and then drop below it, and
$\Sigma_s/\gamma_{j-1}$ grows again.  In case~(b) no such restriction is
needed, because there $R$ is bounded below by $\nu>1$ on all of $(0,1)$.
\end{lemma}

\begin{figure}[t]
\centering
\begin{tikzpicture}
\begin{axis}[
  width=0.86\textwidth, height=6.2cm,
  xlabel={$j$}, ylabel={$\log\gamma_j$},
  xmin=0, xmax=400, ymin=-8, ymax=185,
  xtick={0,120,200,280,400},
  xticklabels={$0$,$L$,$\lfloor p_*N\rfloor$,$U$,$N$},
  ytick={0,50,100,150},
  tick label style={font=\small}, label style={font=\small},
  axis lines=left, clip=false,
]
\addplot[very thick,smooth] coordinates {
(0,0.00) (2,9.10) (4,16.43) (6,22.86) (8,28.69) (10,34.05) (12,39.04)
(14,43.73) (16,48.16) (18,52.36) (20,56.36) (22,60.17) (24,63.83)
(26,67.33) (28,70.70) (30,73.94) (32,77.06) (34,80.07) (36,82.98)
(38,85.80) (40,88.52) (42,91.15) (44,93.71) (46,96.18) (48,98.59)
(50,100.92) (52,103.18) (54,105.38) (56,107.52) (58,109.60) (60,111.62)
(62,113.59) (64,115.50) (66,117.36) (68,119.18) (70,120.94) (72,122.66)
(74,124.33) (76,125.96) (78,127.54) (80,129.09) (82,130.59) (84,132.06)
(86,133.49) (88,134.88) (90,136.23) (92,137.55) (94,138.83) (96,140.08)
(98,141.30) (100,142.48) (102,143.63) (104,144.75) (106,145.84)
(108,146.90) (110,147.93) (112,148.93) (114,149.91) (116,150.85)
(118,151.77) (120,152.66) (122,153.52) (124,154.36) (126,155.17)
(128,155.96) (130,156.72) (132,157.45) (134,158.17) (136,158.85)
(138,159.52) (140,160.16) (142,160.77) (144,161.36) (146,161.93)
(148,162.48) (150,163.01) (152,163.51) (154,163.99) (156,164.45)
(158,164.88) (160,165.30) (162,165.69) (164,166.06) (166,166.41)
(168,166.74) (170,167.05) (172,167.34) (174,167.61) (176,167.86)
(178,168.08) (180,168.29) (182,168.48) (184,168.64) (186,168.79)
(188,168.91) (190,169.02) (192,169.10) (194,169.17) (196,169.21)
(198,169.24) (200,169.24) (202,169.23) (204,169.19) (206,169.14)
(208,169.06) (210,168.97) (212,168.85) (214,168.72) (216,168.56)
(218,168.39) (220,168.19) (222,167.97) (224,167.74) (226,167.48)
(228,167.20) (230,166.90) (232,166.58) (234,166.24) (236,165.88)
(238,165.50) (240,165.09) (242,164.67) (244,164.22) (246,163.75)
(248,163.26) (250,162.75) (252,162.21) (254,161.65) (256,161.07)
(258,160.47) (260,159.84) (262,159.19) (264,158.51) (266,157.81)
(268,157.09) (270,156.34) (272,155.57) (274,154.77) (276,153.94)
(278,153.09) (280,152.22) (282,151.31) (284,150.38) (286,149.42)
(288,148.44) (290,147.42) (292,146.38) (294,145.30) (296,144.20)
(298,143.06) (300,141.89) (302,140.69) (304,139.46) (306,138.19)
(308,136.89) (310,135.56) (312,134.19) (314,132.78) (316,131.33)
(318,129.85) (320,128.32) (322,126.76) (324,125.15) (326,123.50)
(328,121.80) (330,120.06) (332,118.28) (334,116.44) (336,114.55)
(338,112.61) (340,110.62) (342,108.57) (344,106.46) (346,104.29)
(348,102.06) (350,99.76) (352,97.39) (354,94.96) (356,92.44) (358,89.85)
(360,87.17) (362,84.40) (364,81.54) (366,78.58) (368,75.51) (370,72.33)
(372,69.03) (374,65.60) (376,62.02) (378,58.29) (380,54.38) (382,50.29)
(384,45.98) (386,41.42) (388,36.59) (390,31.42) (392,25.84) (394,19.74)
(396,12.91) (398,4.90) (399,0.00)
};
\addplot[dashed,gray] coordinates {(120,-8) (120,152.66)};
\addplot[dashed,gray] coordinates {(280,-8) (280,152.22)};
\addplot[dashed,gray] coordinates {(200,-8) (200,169.24)};
\end{axis}
\end{tikzpicture}
\caption{The scale function for $\mathrm{MAJ}_3$ at $N=400$: $\log\gamma_j$
against $j$, with the window of Lemma~\ref{lem:scale}(c) drawn for
$\eta=1/5$, an admissible choice here because
$\kappa'\ge\kappa'(p_*)$ on $[p_*-\eta,p_*+\eta]$, so that $L=120$,
$\lfloor p_*N\rfloor=200$ and $U=280$ for the interior fixed point
$p_*=1/2$.  The curve rises on $[0,\lfloor p_*N\rfloor]$
and falls on $[\lfloor p_*N\rfloor,N]$, with maximum
$\gamma_{200}=e^{169.2}$.  The drop from the maximum to the ends of the
window is a factor $e^{16.6}$ on the left and $e^{17.0}$ on the right;
it is this gap, of order $e^{\Theta(N)}$, that decouples the two basins.}
\label{fig:gamma}
\end{figure}
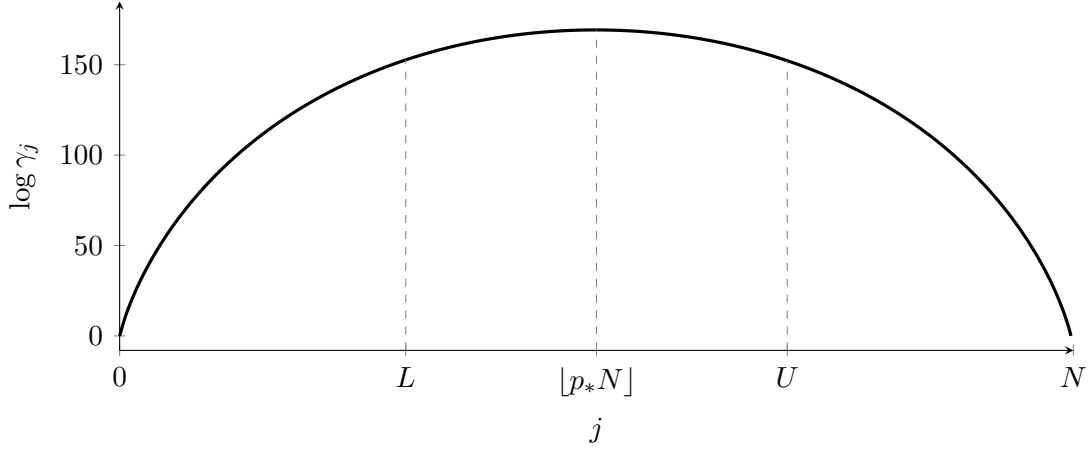

\begin{proof}
A dictator has $D_0=D_1=1$ (Lemma~\ref{lem:endpoint}) and so is degenerate;
hence $f$ is not a dictator, and by Lemma~\ref{lem:endpoint}(iii) at most one
of $D_0,D_1$ is non-zero.  Since also $D_0\ne1$ and $D_1\ne1$, the three
listed cases are exhaustive and mutually exclusive.

\smallskip\noindent\emph{Case (a).}
By Lemma~\ref{lem:endpoint}, $h>0$ on $(0,1)$, so $R<1$ there.  By
Lemma~\ref{lem:rates} with $m_0=1$ and $m_1\ge2$, the function $R$ extends
continuously to $[0,1]$ with $R(0)=\lim_{\pi\to0}\pi/g(\pi)=1/D_0\le1/2$ and
$R(1)=0$.  A continuous function that is $<1$ on a compact interval attains
a maximum $\lambda<1$.  Hence $\gamma_i/\gamma_{i'}\le\lambda^{\,i-i'}$ for
$i\ge i'$, because $\gamma_i/\gamma_{i'}$ is a product of $i-i'$ ratios
$R(l/N)$, each at most $\lambda$, and therefore
$M_j=\sum_{i\ge j}\gamma_i/\gamma_j\le\sum_{n\ge0}\lambda^n=(1-\lambda)^{-1}$.

\smallskip\noindent\emph{Case (b).}
Dually, $h<0$ so $R>1$ on $(0,1)$; here $m_0\ge2$ and $m_1=1$, so
$R(\pi)\asymp\pi^{1-m_0}\to\infty$ as $\pi\downarrow0$ and $R(1)=D_1\ge2$.
Thus $R$ is continuous and $>1$ on $(0,1]$ and tends to $+\infty$ at $0$, so
$\nu:=\inf_{(0,1)}R>1$ and $\gamma_i/\gamma_{i'}\ge\nu^{\,i-i'}$ for
$i\ge i'$.  Then
$\Lambda_j=\sum_{i<j}\gamma_i/\gamma_{j-1}\le\sum_{n\ge0}\nu^{-n}=\nu/(\nu-1)$,
which is~(b1).  Also $\Sigma_N\le\gamma_{N-1}\nu/(\nu-1)$ by the
same estimate, whence
$\Sigma_N-\Sigma_s\ge\gamma_{N-1}\ge\frac{\nu-1}{\nu}\Sigma_N$ for
$s\le N-1$, which is~(b2).  For~(b3), the definition of $\Lambda$
gives
$\frac{\Sigma_s}{\Sigma_j}\Lambda_j=\frac{\Sigma_s}{\gamma_{j-1}}
=\Lambda_s\frac{\gamma_{s-1}}{\gamma_{j-1}}$, and
$\gamma_{s-1}/\gamma_{j-1}=\prod_{l=s}^{j-1}R(l/N)^{-1}\le\nu^{-(j-s)}$ for
$1\le s\le j\le N-1$, since every factor is at most $\nu^{-1}$; with
$\Lambda_s\le\nu/(\nu-1)$ this gives
$\Sigma_s/\gamma_{j-1}\le\frac{\nu}{\nu-1}\nu^{-(j-s)}$.

\smallskip\noindent\emph{Case (c): the window.}
Here $h<0$ on $(0,p_*)$ and $h>0$ on $(p_*,1)$, so $R>1$ on $(0,p_*)$,
$R(p_*)=1$ and $R<1$ on $(p_*,1)$; in particular $R\ge1$ on $(0,p_*]$.
Put $\kappa=-\log R=\log\beta-\log\delta$, the $\Psi'$ of
Section~\ref{sec:potential} and a $C^1$ function on a
neighbourhood of $p_*$ since $\beta$ and $\delta$ are polynomials that are
positive there.  We have $\kappa(p_*)=0$ because $\beta(p_*)=\delta(p_*)$,
and $\kappa'=\beta'/\beta-\delta'/\delta$, so at $p_*$, where
$\beta=\delta$,
$\kappa'(p_*)=(\beta'(p_*)-\delta'(p_*))/\beta(p_*)=h'(p_*)/\beta(p_*)>0$ by
Theorem~\ref{thm:fp}.  By continuity of $\kappa'$ there is $\eta>0$ with
$\kappa'\ge c_1:=\kappa'(p_*)/2$ on $[p_*-\eta,p_*+\eta]$, and the mean value
theorem gives
\begin{equation}\label{eq:kappa}
\log R(\pi)\;\ge\;c_1(p_*-\pi)\ \ \text{on }[p_*-\eta,p_*],
\qquad
\log R(\pi)\;\le\;-c_1(\pi-p_*)\ \ \text{on }[p_*,p_*+\eta].
\end{equation}
Shrinking $\eta$ if necessary we may also assume
$\eta<\min(p_*,1-p_*)/2$.  On $(0,p_*-\eta]$ the function $R$ is continuous,
$>1$, and tends to $+\infty$ at $0$ (Lemma~\ref{lem:rates}, $m_0\ge2$), so
its infimum $\nu$ exceeds $1$.  On $[p_*+\eta,1)$ it is continuous, $<1$,
and extends continuously by $R(1)=0$ since $m_1\ge2$, so its supremum
$\lambda$ is $<1$.

\smallskip\noindent\emph{Proof of (c5).}
From $\eta<\min(p_*,1-p_*)/2$ we get $L\ge(p_*-\eta)N-1\ge\frac{p_*}2N-1$
and $N-U\ge(1-p_*-\eta)N-1\ge\frac{1-p_*}2N-1$, while $L\le p_*N$,
$U\ge p_*N$ and $i_0-L\ge\eta N-1$, $U-i_0\ge\eta N-1$.  Hence for
$N\ge N_0(f)$ all of $L$, $N-L$, $U$, $N-U$ lie between $cN$ and $N$, and
$1\le L<i_0<U\le N-1$.

\smallskip\noindent\emph{Proof of (c1).}
For $j\le L$ and $i<j$ the ratio $\gamma_i/\gamma_{j-1}$ is a product of the
factors $R(l/N)^{-1}$ with $i<l\le j-1$, and each such $l$ satisfies
$l/N\le(L-1)/N<p_*-\eta$, so $R(l/N)\ge\nu$; as in case~(b),
$\Lambda_j\le\nu/(\nu-1)$.  For $j\ge U$ and $i\ge j$ the ratio
$\gamma_i/\gamma_j$ is a product of factors $R(l/N)$ with $j<l\le i$, and
each such $l$ satisfies $l/N>U/N\ge p_*+\eta$, so $R(l/N)\le\lambda$; as in
case~(a), $M_j\le(1-\lambda)^{-1}$.

\smallskip\noindent\emph{Proof of (c2) and (c3).}
Let $1\le s\le j\le L$.  As in case~(b),
$\frac{\Sigma_s}{\Sigma_j}\Lambda_j=\frac{\Sigma_s}{\gamma_{j-1}}
=\Lambda_s\frac{\gamma_{s-1}}{\gamma_{j-1}}$, and
$\gamma_{s-1}/\gamma_{j-1}=\prod_{l=s}^{j-1}R(l/N)^{-1}\le\nu^{-(j-s)}$,
because every $l$ occurring satisfies $l\le j-1\le L-1$ and hence
$l/N<p_*-\eta$, where $R\ge\nu$.  Since $\Lambda_s\le\nu/(\nu-1)$ by~(c1),
this is~(c2).  Taking $j=L$ and using $\Sigma_L\ge\gamma_{L-1}$
gives $\Sigma_s/\Sigma_L\le\Sigma_s/\gamma_{L-1}\le\frac{\nu}{\nu-1}\nu^{-(L-s)}$,
which is~(c3).

\smallskip\noindent\emph{Proof of (c4).}
By~(c1), $\Sigma_L\le\frac{\nu}{\nu-1}\gamma_{L-1}$, while
$\Sigma_N\ge\gamma_{i_0-1}$, so
\[
\frac{\Sigma_L}{\Sigma_N}\;\le\;\frac{\nu}{\nu-1}\cdot\frac{\gamma_{L-1}}{\gamma_{i_0-1}}
\;\le\;\frac{\nu}{\nu-1}\exp\Bigl(-\frac{c_1}{N}\sum_{l=L+1}^{i_0-1}(i_0-l)\Bigr)
\;\le\;\frac{\nu}{\nu-1}\,e^{-c_1\eta^2N/4}\;\le\;\tfrac12
\]
for $N\ge N_0(f)$.  Indeed
$\gamma_{L-1}/\gamma_{i_0-1}=\prod_{l=L}^{i_0-1}R(l/N)^{-1}$; the product
starts at $l=L+1$ in the display because $L/N$ may fall just below
$p_*-\eta$, where~\eqref{eq:kappa} does not bound $\log R$, and the omitted
factor $R(L/N)^{-1}$ is at most $1$ because $R\ge1$ on $(0,p_*]$.  Each
remaining $l$ satisfies $p_*-\eta<(L+1)/N\le l/N\le(i_0-1)/N<p_*$, so
by~\eqref{eq:kappa} and $i_0\le p_*N$ we have
$R(l/N)^{-1}\le e^{-c_1(p_*-l/N)}\le e^{-c_1(i_0-l)/N}$.  Finally
$\sum_{l=L+1}^{i_0-1}(i_0-l)=\binom{i_0-L}{2}\ge\eta^2N^2/4$ for
$N\ge N_0(f)$, because $i_0-L\ge\eta N-1$.

\smallskip\noindent\emph{Proof of (c6).}
We first bound $\Lambda_j$ and $M_j$ inside the window.  Fix $L\le j\le i_0$
and put $a=i_0-j+1\ge1$.  Using $p_*-l/N\ge(i_0-l)/N$
for $l\le i_0$, for $L\le i\le j-1$ and $u=j-1-i\ge0$ we have,
by~\eqref{eq:kappa},
\[
\log\frac{\gamma_i}{\gamma_{j-1}}
=-\sum_{l=i+1}^{j-1}\log R(l/N)
\le-\frac{c_1}{N}\sum_{l=i+1}^{j-1}(i_0-l)
=-\frac{c_1}{N}\Bigl(ua+\frac{u(u-1)}{2}\Bigr).
\]
Discarding the quadratic term gives
$\sum_{u\ge0}e^{-c_1ua/N}\le1+N/(c_1a)$; discarding the linear term gives
$\sum_{u\ge0}e^{-c_1u(u-1)/(2N)}\le2+\int_0^\infty e^{-c_1x^2/(2N)}dx
\le2+\sqrt{2N/c_1}$, by the substitution $x=\sqrt{2N/c_1}\,y$ and
$\int_0^\infty e^{-y^2}dy<1$.  Hence
$\sum_{i=L}^{j-1}\gamma_i/\gamma_{j-1}\le C(1+\min(\sqrt N,N/a))$.  For
$i<L$ we use $R>\nu>1$ on $(0,p_*-\eta]$ and
$\gamma_{j-1}\ge\gamma_{L-1}$ (which holds because $R\ge1$ on
$[L/N,i_0/N]$), giving
$\sum_{i<L}\gamma_i/\gamma_{j-1}\le\sum_{i<L}\gamma_i/\gamma_{L-1}
\le\nu/(\nu-1)$.  Adding the two contributions gives
$\Lambda_j\le C\bigl(1+\min(\sqrt N,N/a)\bigr)$ for $L\le j\le i_0$.

The bound on $M_j$ does not follow from the same computation, because the
inequality just used holds only to the left of $p_*$: that inequality comes
from $i_0=\lfloor p_*N\rfloor\le p_*N$, and for the same reason its mirror
image is false.  What is available on the right is $p_*N<i_0+1$, hence
$l/N-p_*\ge(l-i_0-1)/N$ for every $l$.  Fix $i_0\le j\le U$, put
$b=j-i_0\ge0$, and for $i\ge j$ put $u=i-j\ge0$.  The second bound
of~\eqref{eq:kappa} gives
\[
\log\frac{\gamma_i}{\gamma_j}=\sum_{l=j+1}^{i}\log R(l/N)
\;\le\;-\frac{c_1}{N}\sum_{l=j+1}^{i}(l-i_0-1)
\;=\;-\frac{c_1}{N}\Bigl(u(b-1)+\frac{u(u+1)}{2}\Bigr).
\]
The quantity in parentheses is at least $u(u-1)/2$ for every $b\ge0$, and at
least $u(b-\frac12)\ge ub/2$ when $b\ge1$.  Restricting to $i\le U-1$,
so that every $l$ in the sum satisfies $l/N\in[p_*,p_*+\eta]$, and
summing over $u\ge0$ as before, gives
$\sum_{i=j}^{U-1}\gamma_i/\gamma_j\le C(1+\sqrt N)$ when $b=0$ and
$\le C(1+\min(\sqrt N,N/b))$ when $b\ge1$.  For $i\ge U$ split the
product $\gamma_i/\gamma_j=\prod_{l=j+1}^{i}R(l/N)$ at $l=U$: every $l$
here satisfies $l\ge i_0+1>p_*N$, so the factors with $l\le U$ are at
most $1$ because $R\le1$ on $[p_*,1]$, and those with $l>U$ satisfy
$l/N>p_*+\eta$ and are at most $\lambda<1$, so
$\sum_{i\ge U}\gamma_i/\gamma_j\le(1-\lambda)^{-1}$.  Since
$N/b\le2N/(b+1)$ for $b\ge1$, and since $\min(\sqrt N,N)=\sqrt N$ covers
the case $b=0$, enlarging the constant turns these bounds into
$M_j\le C\bigl(1+\min(\sqrt N,N/(j-i_0+1))\bigr)$.

In the two bounds just proved, $a=i_0-j+1$ ranges over $1\le a\le i_0-L$ as
$j$ runs over $L<j\le i_0$, and $a=j-i_0+1$ ranges over $1\le a\le U-i_0$ as
$j$ runs over $i_0\le j<U$.  Both ranges of $a$ lie in
$1\le a\le A:=\lfloor\eta N\rfloor+2$: indeed $i_0\le p_*N$ and
$L>(p_*-\eta)N-1$ give $i_0-L<\eta N+1\le A$, while
$U<(p_*+\eta)N+1$ and $i_0>p_*N-1$ give $U-i_0<\eta N+2\le A+1$ and hence
$U-i_0\le A$, both quantities being integers.  Therefore
\[
\sum_{L<j\le i_0}\Lambda_j+\sum_{i_0\le j<U}M_j
\;\le\;2C\sum_{a=1}^{A}\Bigl(1+\min\bigl(\sqrt N,\tfrac Na\bigr)\Bigr).
\]
Since
\[
\sum_{a=1}^{A}\min\bigl(\sqrt N,\tfrac Na\bigr)
\le\sum_{a\le\sqrt N}\sqrt N+\sum_{\sqrt N<a\le A}\frac Na
\le N+N\log A ,
\]
the right-hand side of the previous display is at most
$2C\bigl(A+N+N\log A\bigr)$.  Finally $A\le\eta N+2\le N$ and
$\log A\le\log N$ for $N\ge N_0(f)$, because $\eta<1/2$; so the whole is at
most $2C\,N(2+\log N)\le C'N\log N$ for $N\ge N_0(f)$, enlarging $N_0(f)$
once more so that $\log N\ge2$.
\end{proof}

The next lemma is the counterpart of Lemma~\ref{lem:scale} for a degenerate
rule; it owns the residual mean-field map $\tilde g$ and everything derived
from it.

\begin{lemma}[Profile of a degenerate rule]\label{lem:degprofile}
Let $f$ be monotone, non-constant, not a dictator, with $D_0(f)=1$; let
$\tilde f$ be the residual rule supplied by Lemma~\ref{lem:degstruct}, let
$m=m(f)\ge2$ be the least size of a minterm of $\tilde f$, and put
$K=N^{(m-1)/m}$.  There are $\pi_0\in(0,1/2]$, $c_0>0$ and $C<\infty$,
depending only on $f$, such that the following hold for every $N\ge N_0(f)$.
\begin{enumerate}[label=(\roman*),itemsep=2pt]
\item $\gamma_i\le1$ for $0\le i\le N-1$; consequently $\Sigma_j\le j$ for
$0\le j\le N$.
\item $\gamma_i\ge c_0$ for $0\le i\le K+1$; consequently
$\Sigma_N\ge c_0(K-1)$.
\item $\dfrac{\Lambda_j}{\delta_j}\ge N$ for $1\le j\le N-1$, and
$\dfrac{\Lambda_j}{\delta_j}\le C\,N$ for $1\le j\le\lceil K\rceil$.
\item $M_j\le C\,(N/j)^{m-1}$ for $\lceil K\rceil<j\le\pi_0N$, and
$M_j\le C$ for $\pi_0N<j\le N-1$.
\end{enumerate}
\end{lemma}

\begin{proof}
By Lemma~\ref{lem:degstruct}, $D_1(f)=0$ and $f$ equals the disjunction of a
single coordinate with $\tilde f$ applied to the remaining $r-1$; taking
expectations under $\mathrm{Bern}(\pi)^{\otimes r}$ and conditioning on that
coordinate gives $g(\pi)=\pi+(1-\pi)\tilde g(\pi)$ with
$\tilde g(\pi)=\E_{\mathrm{Bern}(\pi)^{\otimes(r-1)}}[\tilde f]$.  As in the
proof of Lemma~\ref{lem:rates}, $\tilde g(\pi)=\pi^m\tilde{\mathsf g}(\pi)$
with $\tilde{\mathsf g}$ continuous and positive on $[0,1]$, so
\begin{equation}\label{eq:gtilde}
c_g\pi^m\;\le\;\tilde g(\pi)\;\le\;C_g\pi^m\qquad(\pi\in[0,1]).
\end{equation}
Also $h(\pi)=(1-\pi)\tilde g(\pi)>0$ on $(0,1)$, so $f$ is in case~(a) of
Corollary~\ref{cor:sign} and $R<1$ on $(0,1)$; in particular $\gamma$ is
decreasing and $\gamma_i\le\gamma_0=1$ for all $i$, whence
$\Sigma_j=\sum_{i<j}\gamma_i\le j$.  This is~(i).  Since
$1-g(\pi)=(1-\pi)(1-\tilde g(\pi))$,
\[
\frac1{R(\pi)}\;=\;\frac{\pi+(1-\pi)\tilde g(\pi)}{\pi(1-\tilde g(\pi))},
\qquad\text{so}\qquad
\log\frac1{R(\pi)}=\log\Bigl(1+\frac{(1-\pi)\tilde g(\pi)}{\pi}\Bigr)
-\log\bigl(1-\tilde g(\pi)\bigr).
\]
Fix $\pi_0\in(0,1/2]$ small enough that $C_g\pi_0^{m-1}\le1$ and
$\tilde g(\pi_0)\le1/2$.  For $\pi\in(0,\pi_0]$, using
$u/2\le\log(1+u)\le u$ for $0\le u\le1$ and $-\log(1-v)\le2v$ for
$0\le v\le1/2$, together with~\eqref{eq:gtilde},
\begin{equation}\label{eq:logR}
c_-\pi^{m-1}\;\le\;\log\frac{1}{R(\pi)}\;\le\;C_+\pi^{m-1},
\qquad c_-=\tfrac12(1-\pi_0)c_g,\ \ C_+=3C_g .
\end{equation}
On $[\pi_0,1)$ the function $R$ is continuous and $<1$, and extends
continuously by $R(1)=0$, so $R\le\lambda<1$ there.

Since $K=N^{(m-1)/m}=o(N)$ and $m\ge2$, there is $N_0(f)$ such that
$K\ge2$ and $\lceil K\rceil\le K+1\le\pi_0N\le N/2$ for $N\ge N_0(f)$, the
last inequality because $\pi_0\le1/2$; these bounds are used repeatedly
below.  Summing~\eqref{eq:logR} and using
$\frac{j^m}{m}\le\sum_{l\le j}l^{m-1}\le j^m$,
\[
\frac{c_-}{m}\cdot\frac{j^m}{N^{m-1}}\;\le\;\log\frac{1}{\gamma_j}
\;\le\;C_+\frac{j^m}{N^{m-1}}\qquad(1\le j\le \pi_0N).
\]
In particular $\gamma_i\ge c_0:=e^{-2^mC_+}$ for $0\le i\le K+1$.  For
$i=0$ this holds because $\gamma_0=1\ge c_0$, the display above having
nothing to say at $i=0$.  For $1\le i\le K+1$ it holds because such
$i$ satisfy $i\le K+1\le\pi_0N$, so that the display applies, and
$i\le2K$ since $K\ge2$, whence $i^m/N^{m-1}\le(2K)^m/N^{m-1}=2^m$.  Consequently, $\gamma$ being
decreasing,
$\Sigma_N\ge\Sigma_{\lfloor K\rfloor}\ge\lfloor K\rfloor\gamma_{\lfloor K\rfloor-1}
\ge c_0(K-1)$, which is~(ii).

For~(iii), $\gamma_i\ge\gamma_{j-1}$ for $i<j$ gives $\Sigma_j\ge j\gamma_{j-1}$,
so by~\eqref{eq:pgqg}
\[
\frac{\Lambda_j}{\delta_j}\;=\;\frac{\Sigma_j}{\delta_j\gamma_{j-1}}
\;=\;\frac{\Sigma_j}{\beta_j\gamma_j}
\;\ge\;\frac{j\gamma_{j-1}}{\delta_j\gamma_{j-1}}\;=\;\frac{j}{\delta_j}\;\ge\;N,
\]
the last step because $\delta_j=\frac jN(1-g(j/N))\le\frac jN$.  For the
second half of~(iii), let $j\le\lceil K\rceil$.  Then $j\le K+1$, so
$\gamma_j\ge c_0$ by~(ii), and $j\le\pi_0N\le N/2$, so $N-j\ge N/2$; with
$\Sigma_j\le j$ from~(i) and $1/\beta_j\le N^2/(j(N-j))$ from
Lemma~\ref{lem:rates}(a),
\[
\frac{\Lambda_j}{\delta_j}\;=\;\frac{\Sigma_j}{\beta_j\gamma_j}
\;\le\;\frac{j}{c_0}\cdot\frac{N^2}{j(N-j)}\;=\;\frac{N^2}{c_0(N-j)}
\;\le\;\frac{2N}{c_0}.
\]

For~(iv), let first $j\le\pi_0N$.  For $j\le i\le\pi_0N$,~\eqref{eq:logR}
gives
$\gamma_i/\gamma_j\le\exp\bigl(-c_-N^{1-m}(i-j)j^{m-1}\bigr)$, so with
$\tau=c_-(j/N)^{m-1}>0$,
\[
\sum_{i=j}^{\lfloor\pi_0N\rfloor}\frac{\gamma_i}{\gamma_j}
\;\le\;\frac{1}{1-e^{-\tau}}\;\le\;1+\frac1\tau
\;\le\;\Bigl(1+\frac1{c_-}\Bigr)\Bigl(\frac Nj\Bigr)^{m-1},
\]
the middle inequality holding for every $\tau>0$ because
$e^{-\tau}\le1/(1+\tau)$, and the last because $j\le N$.
For $i>\pi_0N$ we use $R\le\lambda$ there and
$\gamma_{\lfloor\pi_0N\rfloor}\le\gamma_j$, giving
$\sum_{i>\pi_0N}\gamma_i/\gamma_j\le(1-\lambda)^{-1}$.  Hence
$M_j\le C(N/j)^{m-1}$, which is the first half of~(iv).  If instead
$j>\pi_0N$, then every $i\ge j$ has $\gamma_i/\gamma_j\le\lambda^{\,i-j}$,
so $M_j\le(1-\lambda)^{-1}$.
\end{proof}

The last lemma of the subsection records, once and for all, what the
reflection $i\mapsto N-i$ does.  Both theorems use it to halve their case
analysis, and the third part is what lets the case $s\ge U$ be read off from
the cases already treated.

\begin{lemma}[Duality]\label{lem:duality}
Let $f$ be monotone and non-constant, and let $f^*(x)=1-f(\mathbf 1-x)$,
which is again monotone and non-constant.  Fix $N\ge2$.
\begin{enumerate}[label=(\roman*),itemsep=3pt]
\item $g_{f^*}(\pi)=1-g(1-\pi)$, hence $\beta^{f^*}(\pi)=\delta(1-\pi)$ and
$\delta^{f^*}(\pi)=\beta(1-\pi)$.  The birth-death chain of
Proposition~\ref{prop:bd} for $f^*$ started at $N-s$ is the reflection
$i\mapsto N-i$ of the chain for $f$ started at $s$; consequently
$\varphi_{f^*}(N-s)=\varphi_f(s)$ for $1\le s\le N-1$, and
$\max_{1\le s\le N-1}\varphi_{f^*}(s)=\max_{1\le s\le N-1}\varphi_f(s)$.
\item $D_0(f^*)=D_1(f)$ and $D_1(f^*)=D_0(f)$.  Hence $f^*$ is a dictator if
and only if $f$ is, $f^*$ is non-degenerate if and only if $f$ is, and, when
$f$ is not a dictator, $d_{f^*}(N-s)=d_f(s)$ for $1\le s\le N-1$.  Moreover
$f$ falls in case~(a) of Lemma~\ref{lem:scale} if and only if $f^*$ falls in
case~(b), and $f$ falls in case~(c) if and only if $f^*$ does.
\item Suppose $f$ is in case~(c) of Lemma~\ref{lem:scale}, with interior
fixed point $p_*$.  Then $g_{f^*}$ has the unique interior fixed point
$1-p_*$, and a half-width is admissible for $f^*$ exactly when it is
admissible for $f$.  Moreover, for every half-width $\eta$ of $f$ with
window $i_0,L,U$ as in~\eqref{eq:window}, and for every half-width $\eta^*$
of $f^*$ with window $i_0^*,L^*,U^*$ formed from~\eqref{eq:window} applied
to $f^*$,
\[
s\ge U\ \Longrightarrow\ N-s<U^*\qquad(1\le s\le N-1).
\]
\end{enumerate}
\end{lemma}

\begin{proof}
(i)  The map $x\mapsto\mathbf 1-x$ reverses the coordinatewise order and
$b\mapsto1-b$ reverses $\{0,1\}$, so $f^*$ is monotone; it is non-constant
because $f$ is.  If $X\sim\mathrm{Bern}(\pi)^{\otimes r}$ then
$\mathbf 1-X\sim\mathrm{Bern}(1-\pi)^{\otimes r}$, so
$g_{f^*}(\pi)=\E\bigl[1-f(\mathbf 1-X)\bigr]=1-g(1-\pi)$.  Hence
\[
\begin{aligned}
\beta^{f^*}(\pi)&=(1-\pi)g_{f^*}(\pi)=(1-\pi)\bigl(1-g(1-\pi)\bigr)=\delta(1-\pi),\\
\delta^{f^*}(\pi)&=\pi\bigl(1-g_{f^*}(\pi)\bigr)=\pi\,g(1-\pi)=\beta(1-\pi).
\end{aligned}
\]
By Proposition~\ref{prop:bd} the chain for $f^*$ therefore has
$\beta^{f^*}_i=\delta_{N-i}$ and $\delta^{f^*}_i=\beta_{N-i}$ for
$1\le i\le N-1$, which are exactly the transition probabilities of the
reflected chain, and $i\mapsto N-i$ is a bijection of $\{0,\dots,N\}$
mapping the absorbing set $\{0,N\}$ to itself.  So the two absorption times
have the same law when the starting states correspond under the reflection,
giving $\varphi_{f^*}(N-s)=\varphi_f(s)$; and since $s\mapsto N-s$ permutes
$\{1,\dots,N-1\}$, the two maxima agree.

(ii)  By~\eqref{eq:endpointdeg},
$D_0(f^*)=\#\{i:f^*(e_i)=1\}=\#\{i:f(\mathbf 1-e_i)=0\}=D_1(f)$, and
$D_1(f^*)=D_0(f)$ follows by applying this to $f^*$, since $f^{**}=f$.  The
three assertions about dictators, non-degeneracy and the cases of
Lemma~\ref{lem:scale} are immediate, all four notions being defined by the
pair $(D_0,D_1)$: a dictator is $(1,1)$ by
Theorem~\ref{thm:trichotomy}, non-degeneracy is $D_0\ne1$ and $D_1\ne1$, and
the cases~(a),~(b),~(c) of Lemma~\ref{lem:scale} are $D_0\ge2,D_1=0$;
$D_0=0,D_1\ge2$; and $D_0=D_1=0$, which the swap $(D_0,D_1)\mapsto(D_1,D_0)$
exchanges as claimed.  For $d$, use Definition~\ref{def:attracting}: in
case~(c) both rules have $d(\cdot)=\min(\cdot,N-\cdot)$ and
$\min(N-s,s)=\min(s,N-s)$; if $D_0(f)\ge1$ then $d_f(s)=N-s$ while
$D_1(f^*)\ge1$ gives $d_{f^*}(N-s)=N-s$; and if $D_1(f)\ge1$ then
$d_f(s)=s$ while $D_0(f^*)\ge1$ gives $d_{f^*}(N-s)=N-(N-s)=s$.

(iii)  $g_{f^*}(1-p_*)=1-g(p_*)=1-p_*$, and this is the only interior fixed
point of $g_{f^*}$ by Theorem~\ref{thm:fp}, which applies because $f^*$ is
not a dictator by~(ii).  Since $\min(1-p_*,p_*)=\min(p_*,1-p_*)$, the
constraint $0<\eta^*<\min(1-p_*,p_*)/2$ defining the admissible half-widths
of $f^*$ is the same as for $f$.  Finally $U$ and $U^*$ are the
ceilings~\eqref{eq:window}, so $U\ge(p_*+\eta)N$ and
$U^*\ge(1-p_*+\eta^*)N$, and $s\ge U$ gives
\[
N-s\;\le\;(1-p_*-\eta)N\;<\;(1-p_*)N\;<\;(1-p_*+\eta^*)N\;\le\;U^* ,
\]
the two strict inequalities using only $\eta>0$ and $\eta^*>0$: no relation
between $\eta$ and $\eta^*$ is used.
\end{proof}

\subsection{The non-degenerate theorem}

\begin{theorem}[Non-degenerate rules]\label{thm:bin-nondict}
Let $f:\{0,1\}^r\to\{0,1\}$ be monotone, non-constant and non-degenerate.
There are constants $0<c<C<\infty$, depending only on $f$, such that for
every $N\ge2$ and every $1\le s\le N-1$,
\[
c\,N\bigl(1+\log d_f(s)\bigr)\;\le\;\E[T\mid S_0=s]\;\le\;C\,N\bigl(1+\log d_f(s)\bigr).
\]
In particular $\E[T]=O(N\log N)$ uniformly in the initial state, and for
each fixed $\varepsilon>0$, $\E[T]=\Theta(N\log N)$ uniformly over the
initial states with $d_f(s)\ge\varepsilon N$.
\end{theorem}

\begin{proof}
Absorbing the small $N$ as licensed in Section~\ref{sec:binary}, we may
assume $N\ge N_0(f)$ throughout; every appeal below to
Lemma~\ref{lem:rates}(c) or to Lemma~\ref{lem:scale}(c) is made under this
hypothesis.  We recall the three forms~\eqref{eq:Gleft}--\eqref{eq:Ghat} of
the Green's function, which are used throughout:
\[
G(s,j)=\Bigl(1-\tfrac{\Sigma_s}{\Sigma_N}\Bigr)\tfrac{\Lambda_j}{\delta_j}
\ (j\le s),
\quad
G(s,j)=\tfrac{\Sigma_s}{\Sigma_j}\widehat G(j)\ (j\ge s),
\quad
G(s,j)\le\widehat G(j)\le\min\bigl(\tfrac{\Lambda_j}{\delta_j},\tfrac{M_j}{\beta_j}\bigr).
\]

\smallskip\noindent\emph{Reduction by duality.}
By Lemma~\ref{lem:duality}(i) and~(ii) the dual $f^*(x)=1-f(\mathbf 1-x)$ is
again monotone, non-constant and non-degenerate, and satisfies
$\varphi_{f^*}(N-s)=\varphi_f(s)$ and $d_{f^*}(N-s)=d_f(s)$, while case~(a)
for $f$ is case~(b) for $f^*$.  It is therefore enough to treat cases~(b)
and~(c).

\medskip\noindent\textbf{Case (b): $D_0=0$, $D_1\ge2$.}
Here $d_f(s)=s$, and $m_1=1$, so Lemma~\ref{lem:rates}(b) applies and gives
$1/\delta_j\asymp N^2/(j(N-j))$ for $1\le j\le N-1$; let $\nu>1$ be the
constant of Lemma~\ref{lem:scale}(b).

\emph{Lower bound.}  By~\eqref{eq:Gleft}, $\Lambda_j\ge1$,
$1-\Sigma_s/\Sigma_N\ge c$ (Lemma~\ref{lem:scale}(b2)) and
$1/\delta_j\asymp N^2/(j(N-j))$ (Lemma~\ref{lem:rates}(b)),
\[
\varphi(s)\;\ge\;\sum_{j\le s}G(s,j)
\;\ge\;c\sum_{j\le s}\frac{1}{\delta_j}
\;\ge\;c'\,N\sum_{j\le s}\frac1j
\;=\;c'\,N\,H_s ,
\]
the third step because $N^2/(j(N-j))\ge N/j$.  Since
$H_s\ge\frac12(1+\log s)$ (Lemma~\ref{lem:harmonic}(a)), this is at least
$cN(1+\log s)$.

\emph{Upper bound.}  Split $\varphi(s)=\sum_{j<s}G(s,j)+\sum_{j\ge s}G(s,j)$.
For the first sum,~\eqref{eq:Gleft}, $\Lambda_j\le C$
(Lemma~\ref{lem:scale}(b1)) and $1/\delta_j\asymp N^2/(j(N-j))$
(Lemma~\ref{lem:rates}(b)) give
\[
\sum_{j<s}G(s,j)\;\le\;CN\sum_{j<s}\Bigl(\frac1j+\frac1{N-j}\Bigr)
\;\le\;C'N\bigl(1+\log s\bigr),
\]
using $N^2/(j(N-j))=N(1/j+1/(N-j))$ and, from
Lemma~\ref{lem:harmonic}(a) and~(b), the two bounds
$\sum_{j<s}1/j\le H_s\le1+\log s$ and $\sum_{j<s}1/(N-j)\le2(1+\log s)$.  For the second
sum, the identity in~\eqref{eq:Gright} together with~\eqref{eq:Ghat} gives
$G(s,j)\le\frac{\Sigma_s}{\Sigma_j}\frac{\Lambda_j}{\delta_j}$ for $j\ge s$,
so by $\frac{\Sigma_s}{\Sigma_j}\Lambda_j\le C\nu^{-(j-s)}$
(Lemma~\ref{lem:scale}(b3)) and $1/\delta_j\asymp N^2/(j(N-j))$
(Lemma~\ref{lem:rates}(b)),
\[
\sum_{j\ge s}G(s,j)\;\le\;CN\Bigl(\sum_{j\ge s}\frac{\nu^{-(j-s)}}{j}
+\sum_{j\ge s}\frac{\nu^{-(j-s)}}{N-j}\Bigr)\;\le\;C'N ,
\]
because the first inner sum is at most
$\frac{\nu}{\nu-1}\cdot\frac1s\le\frac{\nu}{\nu-1}$, while the second, on
substituting $v=N-s$ and $u=N-j$, equals $\sum_{u=1}^{v}\nu^{-(v-u)}/u$ and
so is at most $A_\nu$ by Lemma~\ref{lem:harmonic}(c).  Altogether
$\varphi(s)\le CN(1+\log s)$.

\medskip\noindent\textbf{Case (c): $D_0=D_1=0$.}
Here $d_f(s)=\min(s,N-s)$, and $m_0,m_1\ge2$.  Let $\eta$ and $\nu$ be the
constants of Lemma~\ref{lem:scale}(c), let
$i_0=\lfloor p_*N\rfloor$, $L=\lfloor(p_*-\eta)N\rfloor$ and
$U=\lceil(p_*+\eta)N\rceil$ be the window~\eqref{eq:window} of half-width
$\eta$, and apply
Lemma~\ref{lem:rates}(c) with this $\eta$; since $\eta$ depends only on $f$,
so does the threshold $N_0(f,\eta)$ of that lemma, and it is covered by the
$N\ge N_0(f)$ assumed above.  Throughout case~(c),
$1\le L<i_0<U\le N-1$ and $L\asymp N-L\asymp U\asymp N-U\asymp N$
(Lemma~\ref{lem:scale}(c5)); these four relations are used below without
further comment.

\emph{The envelope.}  Split $\sum_{j=1}^{N-1}\widehat G(j)$ over the three
ranges $j\le L$, $L<j<U$ and $j\ge U$.  On $j\le L$ we have $\Lambda_j\le C$
(Lemma~\ref{lem:scale}(c1)) and $1/\delta_j\asymp N/j$
(Lemma~\ref{lem:rates}(c)), so
$\widehat G(j)\le\Lambda_j/\delta_j\le CN/j$ by~\eqref{eq:Ghat} and the first
range contributes at most $CNH_N$.  On $j\ge U$ we have $M_j\le C$
(Lemma~\ref{lem:scale}(c1)) and $1/\beta_j\asymp N/(N-j)$
(Lemma~\ref{lem:rates}(c)), so
$\widehat G(j)\le M_j/\beta_j\le CN/(N-j)$ and the third range likewise
contributes at most $CNH_N$.  On $L<j<U$ both $1/\delta_j\le C$ and
$1/\beta_j\le C$ (Lemma~\ref{lem:rates}(c)), so
by~\eqref{eq:Ghat} and by
$\sum_{L<j\le i_0}\Lambda_j+\sum_{i_0\le j<U}M_j\le CN\log N$
(Lemma~\ref{lem:scale}(c6)),
\[
\sum_{L<j\le i_0}\frac{\Lambda_j}{\delta_j}
+\sum_{i_0<j<U}\frac{M_j}{\beta_j}
\;\le\;C\Bigl(\sum_{L<j\le i_0}\Lambda_j+\sum_{i_0\le j<U}M_j\Bigr)
\;\le\;C'\,N\log N .
\]
Since $H_N\le1+\log N\le2\log N$ for $N\ge N_0(f)$
(Lemma~\ref{lem:harmonic}(a)), we conclude
\begin{equation}\label{eq:envelope}
\max_{1\le s\le N-1}\varphi(s)\;\le\;\sum_{j=1}^{N-1}\widehat G(j)\;\le\;C\,N\log N .
\end{equation}

\emph{Upper bound for $s\le L$.}  By~\eqref{eq:Gleft}, $\Lambda_j\le C$
(Lemma~\ref{lem:scale}(c1)) and $1/\delta_j\asymp N/j$ (Lemma~\ref{lem:rates}(c)),
\[
\sum_{j\le s}G(s,j)\;\le\;CN\sum_{j\le s}\frac1j\;\le\;CN(1+\log s).
\]
For $s<j\le L$ the identity in~\eqref{eq:Gright} together
with~\eqref{eq:Ghat} gives
$G(s,j)\le\frac{\Sigma_s}{\Sigma_j}\cdot\frac{\Lambda_j}{\delta_j}$, and
$\frac{\Sigma_s}{\Sigma_j}\Lambda_j\le C\nu^{-(j-s)}$
(Lemma~\ref{lem:scale}(c2)), so with $1/\delta_j\asymp N/j$
(Lemma~\ref{lem:rates}(c)),
\[
\sum_{s<j\le L}G(s,j)\;\le\;CN\sum_{n\ge1}\frac{\nu^{-n}}{s}\;\le\;CN .
\]
For $j>L$ we use the identity in~\eqref{eq:Gright}, $\Sigma_j\ge\Sigma_L$
and $\frac{\Sigma_s}{\Sigma_L}\le C\nu^{-(L-s)}$
(Lemma~\ref{lem:scale}(c3)):
\[
\sum_{j>L}G(s,j)=\sum_{j>L}\frac{\Sigma_s}{\Sigma_j}\widehat G(j)
\le\frac{\Sigma_s}{\Sigma_L}\sum_{j>L}\widehat G(j)
\le C\nu^{-(L-s)}\cdot N\log N ,
\]
the last step by~\eqref{eq:envelope}.
If $s\le L/2$ then $L-s\ge L/2\ge cN$, and the
right-hand side is at most $1$ for $N\ge N_0(f)$.  If $L/2<s\le L$ then
$1+\log s\ge c\log N$,
and~\eqref{eq:envelope} already gives $\varphi(s)\le CN\log N\le CN(1+\log s)$.
So in all cases $\varphi(s)\le CN(1+\log s)$ for $s\le L$.  Since
$N-s\ge N-L\ge cN$, we have
$1+\log d_f(s)=\min(1+\log s,\,1+\log(N-s))\ge c'(1+\log s)$, and therefore
$\varphi(s)\le C'N(1+\log d_f(s))$.

\emph{Upper bound for $L<s<U$.}  Here $d_f(s)=\min(s,N-s)\asymp N$,
and~\eqref{eq:envelope} gives
$\varphi(s)\le CN\log N\le CN(1+\log d_f(s))$.

\emph{Upper bound for $s\ge U$.}  This follows from the two preceding
paragraphs by Lemma~\ref{lem:duality}: by~(ii) and~(iii) of that lemma $f^*$
is again in case~(c), $\varphi_f(s)=\varphi_{f^*}(N-s)$,
$d_f(s)=d_{f^*}(N-s)$, and $N-s<U^*$ for the window of any half-width
$\eta^*$ used for $f^*$; so $N-s$ lies in one of the two ranges already
treated for $f^*$.

\emph{Lower bound, $s\le L$.}  By~\eqref{eq:Gleft}, $\Lambda_j\ge1$,
$1-\Sigma_s/\Sigma_N\ge1-\Sigma_L/\Sigma_N\ge\frac12$
(Lemma~\ref{lem:scale}(c4), together with $\Sigma_s\le\Sigma_L$) and
$1/\delta_j\asymp N/j$ (Lemma~\ref{lem:rates}(c)),
\[
\varphi(s)\ \ge\ \sum_{j\le s}G(s,j)
\ \ge\ \tfrac12\sum_{j\le s}\frac1{\delta_j}
\ \ge\ c\,N\,H_s\ \ge\ c'\,N(1+\log s)\ \ge\ c'\,N\bigl(1+\log d_f(s)\bigr),
\]
the fourth step by $H_s\ge\frac12(1+\log s)$
(Lemma~\ref{lem:harmonic}(a)) and the last because $d_f(s)\le s$.

\emph{Lower bound, $L<s<U$.}  If $\Sigma_s\le\Sigma_N/2$ then the chain of
inequalities of the previous paragraph, restricted to $j\le L$, gives
$\varphi(s)\ge\frac12\sum_{j\le L}1/\delta_j\ge cNH_L\ge c'N\log N$, using
$L\asymp N$.  Otherwise
$\Sigma_s/\Sigma_N>1/2$ and, by Proposition~\ref{prop:green} for $j\ge s$
together with the definition of $M_j$, with $M_j\ge1$ and
$1/\beta_j\asymp N/(N-j)$ (Lemma~\ref{lem:rates}(c)),
\[
\varphi(s)\;\ge\;\sum_{j\ge U}G(s,j)
=\sum_{j\ge U}\frac{\Sigma_s}{\Sigma_N}\frac{M_j}{\beta_j}
\;\ge\;\tfrac12\sum_{j\ge U}\frac{1}{\beta_j}
\;\ge\;c\,N\sum_{j\ge U}\frac{1}{N-j}\;\ge\;c'N\log N ,
\]
the last step because $\sum_{j\ge U}1/(N-j)=H_{N-U}$ and $N-U\asymp N$.
Since $d_f(s)\le N$, both alternatives give
$\varphi(s)\ge c''N(1+\log d_f(s))$.

\emph{Lower bound, $s\ge U$.}  As for the upper bound, this is dual to the
two ranges just treated: by Lemma~\ref{lem:duality}(ii) and~(iii),
$\varphi_f(s)=\varphi_{f^*}(N-s)$, $d_f(s)=d_{f^*}(N-s)$ and $N-s<U^*$, so
$N-s$ lies in one of them.
\end{proof}

\ignore{
\begin{remark}[Numerics for the non-degenerate case]\label{rem:nondegnum}
The numerical method of Section~\ref{sec:intro} gives
$\varphi(s)/\bigl(N(1+\log d_f(s))\bigr)\in[0.82,1.93]$ uniformly over all
$1\le s\le N-1$ and over $N$ up to $32\,000$, for
$\mathrm{MAJ}_3$, $\mathrm{OR}_2$, $\mathrm{AND}_2$, the threshold-$2$
function on four inputs, and $x_1\vee x_2\vee(x_3\wedge x_4)$.  Over all
$98$ monotone non-constant $f$ with $r=4$ that are non-degenerate and not
dictators, the same ratio lies in $[0.87,1.93]$ at $N=51\,200$.
\end{remark}
}

\subsection{The degenerate theorem}

This subsection treats the two pairs left over by
Theorem~\ref{thm:bin-nondict}, namely $(D_0,D_1)=(1,0)$ and
$(D_0,D_1)=(0,1)$: the rules that are degenerate but not dictators.  By
duality it is enough to treat $D_0=1$, and then $D_1=0$ automatically
(Lemma~\ref{lem:degstruct}).  The estimates it needs were assembled in
Lemma~\ref{lem:degprofile}.

\begin{theorem}[Degenerate rules]\label{thm:bin-deg}
Let $f$ be monotone, non-constant, not a dictator, with $D_0(f)=1$.  Write
$f=x_j\vee\tilde f(x_{-j})$ as in Lemma~\ref{lem:degstruct} and let
$m=m(f)\ge2$ be the least size of a minterm of $\tilde f$.  Then
\[
\max_{1\le s\le N-1}\E[T\mid S_0=s]\;=\;\Theta\bigl(N^{2-1/m}\bigr),
\]
with implied constants depending only on $f$; the lower bound already
holds at $s=\lceil\varepsilon N^{(m-1)/m}\rceil$ for a suitable
$\varepsilon=\varepsilon(f)>0$.  The same holds when
$D_1(f)=1$, by duality.  Consequently the bound
$\E[T]=O(N\log N)$ of Theorem~\ref{thm:bin-nondict} fails for every
degenerate non-dictator $f$, and the exponent $2-1/m$ tends to $2$ as $m$
grows.
\end{theorem}

\begin{proof}
Let $K=N^{(m-1)/m}$ and let $\pi_0$, $c_0$ be as in
Lemma~\ref{lem:degprofile}.  Absorbing the small $N$ as licensed in
Section~\ref{sec:binary}, we may assume $N\ge N_0(f)$ throughout.  By
hypothesis $D_0(f)=1$, so Lemma~\ref{lem:rates}(a) applies as well and gives
$1/\beta_j\le N^2/(j(N-j))$ for $1\le j\le N-1$.  We recall the two
Green's-function forms used below,~\eqref{eq:Gleft} and~\eqref{eq:Ghat}:
\[
G(s,j)=\Bigl(1-\tfrac{\Sigma_s}{\Sigma_N}\Bigr)\tfrac{\Lambda_j}{\delta_j}
\ \ (j\le s),
\qquad
G(s,j)\;\le\;\widehat G(j)\;\le\;
\min\bigl(\tfrac{\Lambda_j}{\delta_j},\ \tfrac{M_j}{\beta_j}\bigr) .
\]

\medskip\noindent\emph{Lower bound.}
Choose $\varepsilon=c_0/4$ and $s=\lceil\varepsilon K\rceil$, which satisfies
$1\le s\le\varepsilon K+1\le N-1$ for $N\ge N_0(f)$, since $K=o(N)$.  By Lemma~\ref{lem:degprofile}(i) and~(ii),
$\Sigma_s\le s\le\varepsilon K+1$ and $\Sigma_N\ge c_0(K-1)$, so
$\Sigma_s/\Sigma_N\le(\varepsilon K+1)/(c_0(K-1))$, which tends to
$\varepsilon/c_0=1/4$ as $K\to\infty$ and so is at most $1/2$ for
$N\ge N_0(f)$.  Therefore, by~\eqref{eq:Gleft} restricted to $j<s$ and
Lemma~\ref{lem:degprofile}(iii),
\[
\varphi(s)\;\ge\;\sum_{j<s}G(s,j)
=\Bigl(1-\frac{\Sigma_s}{\Sigma_N}\Bigr)\sum_{j<s}\frac{\Lambda_j}{\delta_j}
\;\ge\;\tfrac12(s-1)N\;\ge\;\tfrac12(\varepsilon K-1)N\;\ge\;c\,N^{2-1/m}.
\]

\medskip\noindent\emph{Upper bound.}
We bound $\sum_j\widehat G(j)$, which dominates $\varphi(s)$ for every $s$
by~\eqref{eq:Ghat}, over the three ranges of
Lemma~\ref{lem:degprofile}(iii) and~(iv).

For $j\le\lceil K\rceil$:~\eqref{eq:Ghat} and
Lemma~\ref{lem:degprofile}(iii) give $\widehat G(j)\le\Lambda_j/\delta_j\le CN$,
and there are at most $K+1$ such $j$, so
\[
\sum_{j\le\lceil K\rceil}\widehat G(j)\;\le\;C\,N(K+1)\;\le\;C'\,N^{2-1/m}.
\]

For $\lceil K\rceil<j\le\pi_0N$: Lemma~\ref{lem:degprofile}(iv) gives
$M_j\le C(N/j)^{m-1}$, while $j\le\pi_0N\le N/2$ and
Lemma~\ref{lem:rates}(a) give $1/\beta_j\le N^2/(j(N-j))\le2N/j$, so
by~\eqref{eq:Ghat}
\[
\widehat G(j)\;\le\;\frac{M_j}{\beta_j}\;\le\;C'\frac{N^m}{j^m},
\]
whence
\[
\sum_{\lceil K\rceil<j\le\pi_0N}\widehat G(j)\;\le\;C'N^m\int_{K}^{\infty}\frac{dx}{x^m}
\;=\;\frac{C'}{m-1}\,N^mK^{1-m}\;=\;\frac{C'}{m-1}\,N^{2-1/m}.
\]

For $j>\pi_0N$: Lemma~\ref{lem:degprofile}(iv) gives $M_j\le C$, and
Lemma~\ref{lem:rates}(a) gives
$1/\beta_j\le N^2/(j(N-j))\le CN/(N-j)$, so
$\sum_{j>\pi_0N}\widehat G(j)\le CN\,H_N=O(N\log N)$, which is
$o(N^{2-1/m})$ because $2-1/m\ge3/2$.

Adding the three ranges gives
$\max_s\varphi(s)\le\sum_j\widehat G(j)\le C\,N^{2-1/m}$.

\medskip\noindent
The case $D_1(f)=1$ follows by applying the above to $f^*$: by
Lemma~\ref{lem:duality}(ii) it has $D_0(f^*)=1$, and by
Lemma~\ref{lem:duality}(i) its chain is the reflection $i\mapsto N-i$ of that
of $f$, so that $\max_s\varphi_{f^*}(s)=\max_s\varphi_f(s)$.
Finally $N^{2-1/m}/(N\log N)\to\infty$, so the $O(N\log N)$ bound fails.
\end{proof}

\begin{remark}[Numerics for the degenerate case]\label{rem:degnum}
For $f=x_1\vee(x_2\wedge x_3)$, where $m=2$, the computed values are
$\max_s\varphi(s)/N^{3/2}=1.15,\,1.07,\,1.02,\,0.98,\,0.97$ for
$N=800$, $3200$, $12\,800$, $51\,200$ and $204\,800$, and $0.95$ at
$N=819\,200$; the
arguments of the maximum are $29,54,104,203,399,791$, that is,
$\asymp N^{1/2}$ as the theorem predicts.  The rules
$x_1\vee(x_2\wedge x_3\wedge x_4)$ and
$x_1\vee(x_2\wedge\dots\wedge x_5)$, with $m=3$ and $m=4$, match the
exponents $5/3$ and $7/4$ equally well.  Exhaustively over the $162$
monotone non-constant $f$ with $r=4$ that are not dictators, the $64$ with
$D_0=1$ or
$D_1=1$ all have $\max_s\varphi(s)/(N\log N)$ diverging (reaching $84$ at
$N=51\,200$), while the other $98$ stay in $[1.30,1.98]$.
\end{remark}

\begin{remark}[Where the Bernoulli Poincar\'e inequality enters]\label{rem:poincare-role}
Lemma~\ref{lem:poincare} is used for one purpose in the estimates: it proves the trichotomy
of Corollary~\ref{cor:sign}, that $h$ has at most one zero in $(0,1)$ and
that such a zero is repelling.  This is what makes $h$, and hence
$\log R=\log(\delta/\beta)$, of constant sign on each basin, which is exactly the
hypothesis under which Lemma~\ref{lem:scale} bounds
$\gamma$ and $\Sigma$.  The inequality says nothing about the behaviour at
the absorbing endpoints: that is the separate combinatorial content of
Lemmas~\ref{lem:endpoint} and~\ref{lem:rates}.  It is $D_0,D_1$, not the
fixed point, that decide between $\Theta(N\log N)$ and
$\Theta(N^{2-1/m})$.  The fixed point enters the estimates once more,
through the maximum of $\gamma$ shown in Figure~\ref{fig:gamma}:
that maximum is what decouples the two basins, through the factor
$\Sigma_s/\Sigma_L$ in the proof of Theorem~\ref{thm:bin-nondict}.
\end{remark}

\begin{example}[The smallest degenerate non-dictator]\label{ex:smallest}
For $r\le2$ every monotone non-constant $f$ is a dictator, $\mathrm{OR}_2$
or $\mathrm{AND}_2$, and the latter two have $(D_0,D_1)=(2,0)$ and $(0,2)$,
so all are dictators or non-degenerate.  At $r=3$ there are $15$ monotone
non-constant non-dictator functions, exactly $6$ of them degenerate, and,
up to permuting coordinates and dualising, all $6$ equal
$f(x)=x_1\vee(x_2\wedge x_3)$, which has $D_0=1$, $D_1=0$,
$\tilde f(x_2,x_3)=x_2\wedge x_3$ and $m=2$.  Theorem~\ref{thm:bin-deg}
gives $\max_{1\le s\le N-1}\E[T\mid S_0=s]=\Theta(N^{3/2})$, in agreement
with Remark~\ref{rem:degnum}.  This is the smallest arity at which the
$O(N\log N)$ bound of Theorem~\ref{thm:bin-nondict} fails for a
non-dictator.
\end{example}

\section{The dictator case}
\label{sec:dict}

Throughout this section $f(x)=x_j$ for a fixed
coordinate $j\in[r]$, so that $f$ is a dictator in the sense of
Section~\ref{sec:intro}.  A dictator is the only monotone non-constant
$f$ with $g=\mathrm{id}$: such an $f$ has
$\Var_p(f)=p(1-p)=p(1-p)g'(p)$ for every $p\in(0,1)$, so
Lemma~\ref{lem:poincare} holds with equality, forcing $f$ to be a
constant or a dictator.  Hence $h\equiv 0$ and the chain $S_t$ is a
martingale.  Both endpoint degrees
equal $1$, so a dictator is degenerate and
Theorem~\ref{thm:bin-nondict} does not apply to it.

\begin{theorem}[Dictator case]\label{thm:bin-dict}
Let $f(x)=x_j$.  Then for $S_0=s\in\{1,\dots,N-1\}$,
\[
\varphi(s)\;:=\;\E[T\mid S_0=s]\;=\;N(N-s)\bigl(H_{N-1}-H_{N-s-1}\bigr)
\;+\;Ns\bigl(H_{N-1}-H_{s}\bigr).
\]
\end{theorem}

\begin{proof}
Under $f(x)=x_j$ the update rule reads $X_V^{(t+1)}=X_{J_j}^{(t)}$, so
the chain is the asynchronous voter model on the complete graph $K_N$:
at each step a uniformly chosen agent copies the state of a uniformly
chosen agent.  In particular $g(\pi)=\pi$ for all $\pi\in[0,1]$, so by
Remark~\ref{rem:dictcheck} the birth-death description of
Section~\ref{sec:binary} has $\beta_i=\delta_i=i(N-i)/N^{2}$ for
$1\le i\le N-1$, with $0$ and $N$ absorbing, and a flat scale function
$\gamma_i\equiv 1$, $\Sigma_i=i$, $\Sigma_N=N$.  Substituting into
Proposition~\ref{prop:green}, and writing
$l$ for the summation index to keep it apart from the dictator
coordinate $j$,
\[
\varphi(s)=\sum_{l=1}^{N-1}G(s,l),\qquad
G(s,l)=\frac{\Sigma_{s\wedge l}\bigl(\Sigma_N-\Sigma_{s\vee l}\bigr)}{\Sigma_N}
\cdot\frac{1}{\beta_l\gamma_l}
=\frac{(s\wedge l)\bigl(N-(s\vee l)\bigr)}{N}\cdot\frac{N^{2}}{l(N-l)} .
\]
Equivalently, $\varphi$ solves the Poisson equation
\[
\varphi(s+1)+\varphi(s-1)-2\varphi(s)=-\frac{N^{2}}{s(N-s)},
\qquad \varphi(0)=\varphi(N)=0,
\]
whose Green's function for the discrete Laplacian on $\{1,\dots,N-1\}$
with Dirichlet boundary conditions is
$G_0(s,l)=(s\wedge l)(N-(s\vee l))/N$.

Splitting the sum at $l=s$ and simplifying each range,
\[
G(s,l)=\frac{N(N-s)}{N-l}\ \ (l\le s),
\qquad
G(s,l)=\frac{Ns}{l}\ \ (l>s),
\]
so
\begin{align*}
\varphi(s)
&=N(N-s)\sum_{l=1}^{s}\frac{1}{N-l}\;+\;Ns\sum_{l=s+1}^{N-1}\frac{1}{l}\\
&=N(N-s)\bigl(H_{N-1}-H_{N-s-1}\bigr)+Ns\bigl(H_{N-1}-H_{s}\bigr).
\qedhere
\end{align*}
\end{proof}

\begin{corollary}[Asymptotics]\label{cor:dict-asymp}
Let $p_0=s/N$.
\begin{enumerate}[label=(\alph*),itemsep=0pt]
\item $\varphi(s)/N^{2}=\Ent(p_0)+O(1/N)$ uniformly over $1\le s\le N-1$,
with implied constant $3$; in particular $\varphi(s)/N^2\to\Ent(p)$ for
fixed $p\in(0,1)$ and $s=\lfloor pN\rfloor$, and $\varphi(s)=\Theta(N^{2})$
uniformly over the $s$ with $p_0$ bounded away from $\{0,1\}$.
\item $\varphi(1)=NH_{N-1}=\Theta(N\log N)$, and symmetrically
$\varphi(N-1)=NH_{N-1}$.
\item More generally $\varphi(s)=\Theta\bigl(Ns\log(N/s)\bigr)$ in the
regime $1\le s=o(N)$.
\item There are absolute constants $0<c\le C<\infty$ such that
\[
c\,N^{2}\Ent(p_0)\;\le\;\varphi(s)\;\le\;C\,N^{2}\Ent(p_0)
\qquad\text{for every }N\ge2\text{ and every }1\le s\le N-1;
\]
one may take $c=1/6$ and $C=2$.
\end{enumerate}
\end{corollary}

\begin{proof}
For (a) we make one use of the comparison of a sum with an integral: for
integers $1\le a\le b$, the sum being empty when $a=b$,
\begin{equation}\label{eq:harmcomp}
0\;\le\;\sum_{i=a}^{b-1}\frac1i-\log\frac ba\;\le\;\frac1a,
\end{equation}
the left inequality because $1/i\ge\int_i^{i+1}dx/x$ and the right
because $\sum_{i=a+1}^{b-1}1/i\le\int_a^{b-1}dx/x\le\log(b/a)$.  Applying~\eqref{eq:harmcomp} with $(a,b)=(N-s,N)$
gives
$\bigl|H_{N-1}-H_{N-s-1}-\log\frac{N}{N-s}\bigr|\le\frac1{N-s}$, and
with $(a,b)=(s+1,N)$ it gives
$\bigl|H_{N-1}-H_{s}-\log\frac{N}{s+1}\bigr|\le\frac1{s+1}$, whence
$\bigl|H_{N-1}-H_{s}-\log\frac Ns\bigr|\le\frac1{s+1}+\log(1+\frac1s)
\le\frac2s$.  Multiplying the first estimate by $N(N-s)\le N^2$ and the
second by $Ns\le N^{2}$,
\[
\frac{\varphi(s)}{N^{2}}=(1-p_0)\log\frac{1}{1-p_0}+p_0\log\frac1{p_0}
+O(1/N)=\Ent(p_0)+O(1/N),
\]
with implied constant $3$, uniformly over $1\le s\le N-1$.
For (b), take $s=1$: the first term is
$N(N-1)(H_{N-1}-H_{N-2})=N(N-1)/(N-1)=N$ and the second is
$N(H_{N-1}-H_1)=NH_{N-1}-N$, so $\varphi(1)=NH_{N-1}$.  The value at
$s=N-1$ follows by the symmetry $\varphi(s)=\varphi(N-s)$, which is
visible in the formula.  For (c), the claim follows from the same two
expansions: for $1\le s=o(N)$ the first term is
$N(N-s)\bigl(s/N+O(s^{2}/N^{2})\bigr)
=\Theta(Ns)$ and the second is $Ns\log(N/s)\bigl(1+o(1)\bigr)$, which
dominates.

For (d), both $\varphi$ and $N^{2}\Ent(\cdot/N)$ are invariant under
$s\mapsto N-s$, so we may assume $s\le N/2$ and put $t=N-s\ge N/2$.
Write $\varphi(s)=A_1+A_2$ and $N^{2}\Ent(p_0)=Ns\log\frac Ns+\Xi$, where
\[
A_1=Nt\sum_{i=t}^{N-1}\frac1i,\qquad
A_2=Ns\sum_{i=s+1}^{N-1}\frac1i,\qquad
\Xi=Nt\log\frac Nt .
\]
The sum in $A_1$ has $s$ terms, each between $1/(N-1)$ and $1/t$, so
\begin{equation}\label{eq:dictA}
\tfrac12Ns\;\le\;ts\;\le\;A_1\;\le\;Ns .
\end{equation}
Comparing the sum in $A_2$ with $\int_s^{N-1}dx/x$ from above and with
$\int_{s+1}^{N}dx/x$ from below, and using $s+1\le 2s$,
\begin{equation}\label{eq:dictB}
Ns\Bigl(\log\frac Ns-\log2\Bigr)\;\le\;A_2\;\le\;Ns\log\frac Ns ,
\end{equation}
both sides being non-negative because $s\le N/2$.  Finally
$\log\frac Nt=-\log(1-\frac sN)$ and $x\mapsto-\log(1-x)/x$ increases on
$(0,1)$, so for $x=s/N\in(0,\tfrac12]$ we get
$\frac sN\le\log\frac Nt\le\frac{2s\log 2}{N}$ and hence,
with~\eqref{eq:dictA},
\begin{equation}\label{eq:dictTheta}
\tfrac12Ns\;\le\;ts\;\le\;\Xi\;\le\;2ts\log2\;\le\;\tfrac75Ns .
\end{equation}
For the upper bound,~\eqref{eq:dictA},~\eqref{eq:dictB} and
$Ns\le2\Xi$ give
$\varphi(s)\le Ns+Ns\log\frac Ns\le2\Xi+Ns\log\frac Ns
\le2N^{2}\Ent(p_0)$.
For the lower bound,~\eqref{eq:dictTheta} and~\eqref{eq:dictA} give
$\Xi\le\frac75Ns\le\frac{14}5A_1\le\frac{14}5\varphi(s)$.  If
$\log\frac Ns\ge2\log2$ then~\eqref{eq:dictB} gives
$A_2\ge\frac12Ns\log\frac Ns$, so $Ns\log\frac Ns\le2\varphi(s)$; if
$\log\frac Ns<2\log2$ then $Ns\log\frac Ns\le\frac75Ns\le
\frac{14}5\varphi(s)$ by~\eqref{eq:dictA}.  In both cases
$N^{2}\Ent(p_0)=Ns\log\frac Ns+\Xi\le\frac{28}5\varphi(s)
\le6\varphi(s)$.
\end{proof}

\begin{remark}[The separation is not uniform in the initial state]
\label{rem:dict-not-uniform}
Corollary~\ref{cor:dict-asymp}(b) gives $\varphi(1)=\Theta(N\log N)$,
which is the same order as the upper bound $O(N\log N)$ of
Theorem~\ref{thm:bin-nondict} for non-degenerate $f$.  That order is
attained there: if $D_0(f)\ge 2$, so that the only attracting endpoint
is $1$, then $d_f(1)=N-1$ and the non-degenerate expected consensus
time from $s=1$ is $\Theta(N\log N)$ as well.  The dictator branch and
the non-degenerate branch therefore do not separate at every initial
state.  The gap between them is $\Theta(N/\log N)$ when $p_0$ is
bounded away from $\{0,1\}$, where the dictator time is $\Theta(N^{2})$
by Corollary~\ref{cor:dict-asymp}(a) and the non-degenerate time is
$\Theta(N\log N)$.  Any statement of the two branches as an
unconditional dichotomy, with constants independent of the initial
state, is false.
\end{remark}

\begin{remark}[The bulk asymptotic]\label{rem:dict-bulk}
The bulk asymptotic of $\varphi$ is $N^{2}\Ent(p_0)$ and not
$N^{2}p_0(1-p_0)$.  The two agree up to constants for $p_0$ bounded away
from $\{0,1\}$, and differ by a logarithmic factor in the boundary
regime $p_0=\Theta(1/N)$, where $N^{2}\Ent(1/N)=\Theta(N\log N)$ while
$N^{2}p_0(1-p_0)=\Theta(N)$.  Theorem~\ref{thm:bin-dict} has been
verified in exact rational arithmetic against the linear system for all
$N\le 14$ and all $s\in\{1,\dots,N-1\}$.  The ratio
$\varphi(s)/\bigl(N^{2}\Ent(s/N)\bigr)$ of
Corollary~\ref{cor:dict-asymp}(d) is computed to lie in $[0.72,1)$ over
all $2\le N\le 3000$ and all $1\le s\le N-1$, its minimum being attained
at $N=2$.  For the parallel result for the continuous-time voter model
on $K_N$ see \cite[Ch.~14]{AldousFill}.
\end{remark}

\end{document}